\documentclass[a4paper, 11pt]{amsart}
\usepackage[T1]{fontenc}
\usepackage[utf8]{inputenc}
\usepackage{amsmath}
\usepackage{amsthm}
\usepackage{amsfonts}
\usepackage{amssymb}
\usepackage{graphicx}
\usepackage[english]{babel}
\usepackage{version} 
\usepackage{nicefrac}
\usepackage{tipa} 
\usepackage{hyperref}
\usepackage{soul}
\usepackage{enumitem}
\usepackage{doc}
\usepackage{xcolor}

\theoremstyle{definition}
\newtheorem{defin}{Definition}[section]

\theoremstyle{plain}
\newtheorem{theo}[defin]{Theorem}
\newtheorem{lemma}[defin]{Lemma}
\newtheorem{obs}[defin]{Remark}
\newtheorem*{obs*}{Remark}
\newtheorem{prop}[defin]{Proposition}

\newtheorem{theorem}{Theorem}
\newtheorem*{theorem*}{Theorem}
\newtheorem{corollary}[theorem]{Corollary}

\usepackage{xcolor}
\definecolor{light-gray}{gray}{0.1}

\makeatletter
\def\@setthanks{\vspace{-\baselineskip}\def\thanks##1{\@par##1\@addpunct.}\thankses}
\makeatother

\title[Finiteness of CAT$(0)$ group actions]
{Finiteness of CAT$(0)$ group actions
}
\author{Nicola Cavallucci}
\thanks{N. Cavallucci has been partially supported by the SFB/TRR 191, funded by the DFG}
\author{Andrea Sambusetti}
\thanks{A. Sambusetti is member of GNSAGA and acknowledges the support of INdAM during the preparation of this work.}
 
\date{\today}

\begin{document}
	\maketitle

\vspace{-6mm}

\begin{abstract}
We prove some finiteness results for  discrete isometry groups $\Gamma$ of uniformly packed CAT$(0)$-spaces $X$ with uniformly bounded codiameter (up to group isomorphism), and for CAT$(0)$-orbispaces $M = \Gamma \backslash X$ (up to equivariant homotopy equivalence or equivariant diffeomorphism); these results generalize, in nonpositive curvature, classical finiteness theorems of Riemannian geometry. As a corollary, the order of every torsion subgroup of $\Gamma$ is bounded above by a universal constant only depending on the packing constants and the codiameter. The main tool is a splitting theorem for sufficiently collapsed actions: namely we show that if a geodesically complete, packed, CAT$(0)$-space admits a discrete, cocompact group of isometries with sufficiently small systole then it necessarily splits a non-trivial Euclidean factor.
\end{abstract}	

 \tableofcontents
	
\vspace{-6mm}

	\section{Introduction}
	
%In this  paper 
This is the first of two papers devoted to
%where we initiate the study
 the theory of convergence for groups $\Gamma$ acting geometrically (that is discretely, by isometries and with compact quotient) on CAT$(0)$-spaces $X$. 
 In this one, we will mainly focus on finiteness and splitting results. 
%The purpose of this manuscript is to prove finiteness results for {\em \textup{CAT}$(0)$-orbispaces}\footnote{We will use this term as a shortening for {\em quotient of a (proper, geodesically complete) \textup{CAT}$(0)$-space $X$ by a discrete, isometry group $\Gamma$}; when restricting our attention to groups $\Gamma$ acting {\em rigidly} on $X$ (i.e. such that every $g \in \Gamma$ acting as the identity on an open subset    is trivial) then this is equivalent to the notion of  {\em rigid, developable orbispace} as defined in \cite[Ch.11]{dlHG90}  (with CAT$(0)$ universal covering in the sense of orbispaces).}
In our  work, all CAT$(0)$-spaces $X$ are assumed to be  {\em proper}  and {\em geodesically complete}, which ensures many desirable geometric properties, such as the equality of topological dimension  and Hausdorff dimension, the existence of a canonical measure $\mu_X$, etc. (see \cite{Kl99}, \cite{LN19} and Section \ref{sec-CAT} for fundamentals on CAT$(0)$-spaces).
\vspace{2mm}

The following is the first main finding of this paper:

%	\begin{theorem}
%		\label{theo-intro-finiteness}
%		Let $P_0,r_0,D_0 \!> \!0$. Then there are a finite number of isomorphism types of groups acting faithfully, discretely, non-singularly and $D_0$-cocompactly by isometries on a {\color{red}proper}, geodesically complete, $(P_0,r_0)$-packed, $\textup{CAT}_0$-space.
%	\end{theorem}

	\begin{theorem}[Finiteness]
		\label{theo-intro-finiteness} ${}$\\
Given $P_0,r_0,D_0 > 0$, there exist only finitely many groups $\Gamma$
acting 
%faithfully, 
discretely and nonsingularly   by isometries on some proper, geodesically complete, $(P_0,r_0)$-packed, $\textup{CAT}(0)$-space   with quotient of diameter at most  $D_0$, \linebreak  up to isomorphism of abstract groups.  \\
{\em (We stress the fact that, in the above theorem, the  CAT$(0)$-spaces  on which the groups $\Gamma$ are supposed to act are not fixed a-priori).}
\end{theorem}

% {\color{blue} ho levato faithfully, perch\'e nonsingular $\Rightarrow$ faithful}

    Here is a quick explanation of the terms used in the statement above. \\
    \noindent We say that a   group  $\Gamma$ acts {\em nonsingularly} on $X$ if there exists at least one point $x\in X$ such that $\text{Stab}_\Gamma(x)$ is trivial (in particular, the action of $\Gamma$ on $X$ is faithful).
 This condition has many natural consequences (e.g., the existence of a fundamental domain)
and is a  mild assumption on the action, which is  automatically satisfied for instance when the action is faithful and either the group  is torsion-free or $X$ is a homology manifold (see Section \ref{sec-bounds} and \cite{CS22}). 
 We  stress the fact  that nonsingularity is   essential for Theorem \ref{theo-intro-finiteness}:   in \cite{BK90}, 
 %H. Bass and R. Kulkarni   
 Bass and  Kulkarni    exhibit  an infinite family of non-isomorphic, discrete  groups $\Gamma_j$ acting (singularly) on the same CAT$(0)$ space $X$ (a regular tree with bounded valency) with diam$(\Gamma_j \backslash X) \leq D_0$ for all $j$; we thank P.E. Caprace for pointing out to us this example  (see Remark \ref{rem_Bass}). \linebreak
We will   say that  $\Gamma$ acts {\em $D_0$-cocompactly} on $X$ if   diam$(\Gamma \backslash X) \leq D_0$.

\noindent On the other hand,  a metric space $X$ is {\em $(P_0,r_0)$-packed} if all balls $B_X (x,3r_0)$ of radius $3r_0$ in $X$ contain  at most $P_0$ points that are $2r_0$-separated. \linebreak
This condition should be thought as weak form of lower curvature bound at scale $r_0$;
%	however, for Riemannian manifolds or metric measure spaces, it is much weaker than assuming curvature bounded below in the sense of Alexandrov or even a  lower  bound on the (Riemannian or synthetic) Ricci curvature, in particular weaker than the $CD(K,N)$ condition).
however, it is much weaker than assuming the curvature bounded below in the sense of Alexandrov,  or  a  lower  bound on the Ricci curvature for Riemanniann manifolds, or  the CD$(k,n)$ condition for metric measure spaces.
Indeed, the Bishop-Gromov's comparison theorem for Riemannian $n$-manifolds with Ric$_X \geq - (n-1)k$  (or   its generalization to  CD$(k,n)$ spaces) yields a doubling condition  
 \begin{equation}\label{eq:doubling}\frac{\mu (B_X (x,2r)) }{\mu (B_X (x,r))} \leq  C(n,k,R) \hspace{5mm} \end{equation}
 %\forall x \in X \mbox{ and } \forall r \leq R
for the measure of all $r$-balls with $ r \leq R$ (for the Riemannian measure in case of manifolds, and for the reference measure in case of CD$(k,n)$-spaces, cp. for instance \cite{Stu06}),
from which the packing condition at scale $r_0 \leq r/4$ easily follows.
%(cp. Remark 5.1, \cite{CavS20}).
See  \cite{BCGS17}   for a detailed comparison of packing, doubling, entropy and curvature conditions.
Let us just recall here that the packing condition  for a CAT$(0)$-space $X$ is also strictly weaker than the doubling condition   \eqref{eq:doubling} for the natural measure $\mu_X$ since  a local  doubling condition implies that the space has {\em pure dimension} (i.e. the tangent cones at all points $x$ have the same geometric dimension $n$, see \cite[Theorem C]{CavS20}), which restricts considerably the class of spaces under consideration.
%Petrunin: alexandrov m-dimensional implica RCD(k,m)
%	This condition is a weak, metric  substitute of a  lower bound of the Ricci curvature; 	however, it  is much weaker than a true Ricci curvature lower bound for Riemannian manifolds, or than assuming curvature bounded below in the sense of Alexandrov (see \cite{BCGS17},  for a comparison of packing, entropy and curvature conditions).  \linebreak
	
	  Moreover, for the spaces we are interested in, a packing condition  is a natural, and in some sense minimal, assumption. Indeed, every metric space $X$ with a compact quotient is $(P_0, r_0)$-packed for suitable constants $P_0,r_0$ (cp. \cite[Proof of Lemma 5.4]{Cav21ter}).
Namely, 	let  $\textup{CAT}_0(D_0)$  be the class of all isometric actions $\Gamma \curvearrowright X$
where $X$ is a  proper, geodesically complete, CAT$(0)$-space,  and $\Gamma$ is a $D_0$-cocompact, {\em discrete} subgroup of  $\text{Isom}(X)$; then the $(P_0, r_0)$-packing conditions define a filtration 
$$\textup{CAT}_0(D_0) = \bigcup_{P_0, r_0} \textup{CAT}_0(P_0,r_0,D_0)$$
where $\textup{CAT}_0(P_0,r_0,D_0)$ is the subset  of   $\textup{CAT}_0(D_0)$ made of the 
actions $\Gamma \curvearrowright X$  such that  $X$ is, moreover,   $(P_0,r_0)$-packed.

%{\color{blue} \noindent Non mi convince  la notazione $(X, \Gamma)$ per un'{\em azione}: o scriviamo $\Gamma \curvearrowright X$, oppure se vuoi lasciare $(X, \Gamma)$ \`e meglio chiamarle {\em pairs} (comunque dopo le chiamiamo ancora ``pairs'' --in cyan-- va uniformato). \\
%comunque fissiamo una notazione e un termine che vadano bene anche per il secondo articolo...}

%(besides the fact that any geodesically complete metric space with two-sided curvature bounds in the sense of Alexandrov actually is a $C^{1,\alpha}$ Riemannian manifold, by a theorem of Berestovskij and Nikolaev \cite{BN93});
\noindent As proved in  \cite{CavS20},  for geodesically complete CAT$(0)$-spaces, a packing condition  at some scale $r_0$ is equivalent to a uniform upper bound of the canonical measure of all $r$-balls (a condition sometimes called {\em macroscopic scalar curvature bounded below}  cp. \cite{Gut10}, \cite{Sab20}). 
%In \cite{CavS20} is proved that, for complete and geodesically complete CAT(0)-spaces, a packing condition at some scale $r_0$ yields an explicit, uniform control of the packing function at any other  scale $r$: therefore, for these spaces, this condition is  equivalent to similar conditions which have been considered by other authors with different names (``uniform compactness of the family of $r$-balls'' in \cite{Gr81}; ``geometrical boundedness'' in \cite{DY}, etc.). 
Namely, there exist functions $v,V \colon (0,+\infty) \to (0,+\infty)$ depending only on $P_0,r_0$ such that for all $x\in X$ and $R> 0$ we have (cp. Proposition \ref{prop-packing}):
\begin{equation}
\label{eq-intro-volume-estimate}
			v(R) \leq \mu_X(B(x,R)) \leq V(R).
\end{equation}
Moreoever, for geodesically complete CAT$(0)$-spaces
%the dimension of $X$ is bounded in terms of $(P_0,r_0)$, and See Section \ref{subsection-packing-margulis} for more details about the packing condition on CAT$_0$-spaces.
%Although much weaker than a curvature bound, 
the $(P_0,r_0)$-packing condition yields an important generalization  of the classical Margulis' Lemma, due to Breuillard-Green-Tao: {\em there exists a constant $\varepsilon_0 (P_0,r_0) \!>0$ \linebreak
 %only depending on the packing constants $P_0,r_0$,
		%%%(only depending on the $n$-dimensional macroscopic scalr curvature bound) 
such that for every discrete group of isometries $\Gamma$ of   $X$ 
		%%%which is  $P_0$-packed at scale $r_0$, 
the 	$\varepsilon_0$-almost stabilizer $\Gamma_{\varepsilon_0} (x)$ of any point $x$ is virtually nilpotent}  (cp. \cite{BGT11} and Section \ref{subsection-packing-margulis} for details).
	%%% that is, the elements of $\Gamma$ which displace $x$ less than $\varepsilon_0$ generate a  virtually nilpotent group.
We call this $\varepsilon_0=\varepsilon_0 (P_0, r_0)$ the {\em Margulis constant}, since it plays the role of the classical Margulis constant in our metric setting.
\vspace{2mm}	
 
%\noindent On the other hand, we say that a   group  $\Gamma$ acts {\em nonsingularly} on $X$ if there exists at least one point $x\in X$ such that $\text{Stab}_\Gamma(x)$ is trivial (in particular, the action of $\Gamma$ on $X$ is faithful).
% This condition has many natural consequences (for instance, the existence of a fundamental domain, see Section \ref{sec-bounds})
%and is a  mild assumption on the action, which is  automatically satisfied for instance when the action is faithful and either the group  is torsion-free or $X$ is a homology manifold (see Section \ref{sec-bounds} and \cite{CS22}). 
% We  stress the fact  that nonsingularity is   essential for Theorem \ref{theo-intro-finiteness}:   in \cite{BK90}, H. Bass and R. Kulkarni   exhibit  an infinite family of non-isomorphic, discrete  groups $\Gamma_j$ acting (singularly) on the same CAT$(0)$ space $X$ (a regular tree with bounded valency) with diam$(\Gamma_j \backslash X) \leq D_0$ for all $j$; we thank P.E. Caprace for pointing out to us this example {\color{cyan}(see Remark \ref{rem_Bass})}.\\
 %Finally, we will  say that  $\Gamma$ acts {\em $D_0$-cocompactly} on $X$ if   diam$(\Gamma \backslash X) \leq D_0$.
%  \vspace{2mm}
 
%%NOTA: "cocompact" segue da diam(X/Gamma)<D perche' X e' proprio
%% => palla B chiusa di raggio D in X e' compatta 
%% => proiezione B/Gamma pure (immagine di compatto)

	We will see that the finiteness up to isomorphism  of Theorem \ref{theo-intro-finiteness} 
%and Corollary \ref{cor-intro-finiteness-fund-groups} 
can be improved to finiteness up to {\em equivariant homotopy equivalence} of the pairs $(X,\Gamma)$
 (see Section \ref{sec-finiteness-spaces} for the precise definition), and  then deduce from Theorem \ref{theo-intro-finiteness} 
%and of the non-singularity of torsion-free groups $\Gamma$, or of any isometric action  of $\Gamma$ on a homology manifold $X$, 
 corresponding finiteness results for the quotient spaces $M=\Gamma \backslash X$. \linebreak
% Namely, let us denote by
%\vspace{-4mm}
%$$ \textup{CAT}_0(P_0, r_0, D_0)$$   the class of geometric {\color{cyan}actions $\Gamma \curvearrowright X$} we are interested in: that is, the set of  {\color{cyan}pairs $(X,\Gamma)$}, where $X$ is a  proper, geodesically complete,  $\textup{CAT}(0)$-space \linebreak 
%which is $(P_0,r_0)$-packed, and $\Gamma$ is a discrete, $D_0$-cocompact group of  isometries of $X$.
We will call  such spaces
 %  $M = \Gamma \backslash X$  a  
  {\em \textup{CAT}$(0)$-orbispaces}
  \footnote{When restricting our attention to groups $\Gamma$ acting {\em rigidly} on $X$ (i.e. such that every $g \in \Gamma$ acting as the identity on an open subset    is trivial) then this definition is equivalent to the notion of  {\em rigid, developable orbispace} as defined in \cite[Ch.11]{dlHG90}  (with CAT$(0)$ universal covering in the sense of orbispaces). However, our results apply to all nonsingular CAT$(0)$-orbispaces.}, 
  using  this term as a shortening for {\em quotient of a (proper, geodesically complete) \textup{CAT}$(0)$-space $X$ by a discrete, isometry group $\Gamma$};
  we will also say that the CAT$(0)$-orbispace  $M$  is {\em nonsingular}  if $\Gamma$ acts nonsingularly on $X$.
  Notice that  if $\Gamma$ is torsion-free then  $M$   is  a locally CAT$(0)$-space. \\
	%and that $M$ is a  {\em CAT$_0$-homology orbifold} if, moreover, $X$ is supposed to be a homology manifold (with a compatible topology). 
For the following,  let us define the class of all CAT$(0)$-orbispaces  $M =\Gamma  \backslash X$ 
\vspace{-7mm}

 $${\mathcal O}\textup{-CAT}_0(P_0, r_0, D_0) $$ 
 where $\Gamma \curvearrowright X$ is in CAT$_0(P_0, r_0, D_0)$, and   the subclasses 
\vspace{-3mm}

$${\mathcal L}\textup{-CAT}_0(P_0, r_0, D_0) \hspace{1mm} \subset \hspace{1mm} 
 {\mathcal{NS} }\textup{-CAT}_0(P_0, r_0, D_0)  $$

%{\mathcal H}{\mathcal O}\textup{-CAT}_0(P_0, r_0, D_0) 
\noindent of  those which are, respectively,  locally CAT$(0)$-spaces and  non-singular. \\
Then we obtain:

\begin{corollary} 
		\label{cor-intro-finiteness-of-groups}
		There are   finitely many 
		%locally  $\textup{CAT}_0$-
		spaces in the class ${\mathcal L}\textup{-CAT}_0(P_0, r_0, D_0)$ up to homotopy equivalence, and finitely many  spaces in $ {\mathcal{NS}}\textup{-CAT}_0(P_0, r_0, D_0)$, up to equivariant homotopy equivalence. \\
{\em(Two orbispaces $M=\Gamma \backslash X$ and $M'=\Gamma' \backslash X'$ are {\em equivariantly homotopy equivalent} if there exists a group isomorphism $\varphi: \Gamma \rightarrow \Gamma'$  and a $\varphi$-equivariant homotopy equivalence $F:X \rightarrow X'$, that is such that $F(g \cdot x) = \varphi(g) \cdot F(x)$ for all $g \in \Gamma$ and all $x \in X$.)}
	\end{corollary}

%{\color{red}  
%	\begin{corollary} 
%		\label{cor-intro-finiteness-of-groups}
%		There are only a finite number    of $\textup{CAT}_0$-homology orbifolds $(\Gamma, X)$ which are $P_0$-packed at scale $r_0$ with codiameter at most $D_0$, up to equivariant homotopy equivalence. 
%	\end{corollary}
%}
%{\color{red}		 
%	\begin{corollary}
%		\label{cor-intro-finiteness-fund-groups}
%		Let $P_0,r_0,D_0 > 0$. Then there are a finite number of isomorphism types of fundamental groups of compact, locally geodesically complete, locally $\textup{CAT}_0$-spaces with diameter at most $D_0$ and whose universal cover is $(P_0,r_0)$-packed.
%	\end{corollary}
%} 

 %	\begin{cor}[Homeomorphic finiteness]
%		\label{cor-intro-finiteness-of-groups}
%		There are only a finite number of isomorphism classes of discrete, isometry groups $\Gamma$ of \textup{CAT}$(0)$-homology manifold $X$ which are $P_0$-packed at scale $r_0$ with codiameter at most $D_0$. There are only a finite number of homeomorphism classes of possible quotients $\Gamma\backslash X$.
%	\end{cor}

%  Corollary \ref{cor-intro-order-elements} gives usual finiteness results for the number of isomorphic classes of groups acting discretely and isometrically with bounded codiameter on packed CAT$(0)$-homology manifolds.  \ 

 A particular case of Theorem \ref{theo-intro-finiteness}, declined in the Riemannian setting, is:

%	\begin{corollary} 
%		\label{cor-intro-finiteness-of-manifolds}
%		For any fixed $n, \kappa$ and $D_0$, there exist only  finitely many diffeomorphism classes of closed $n$-[dimensional Riemannian manifolds $M$ with   sectional curvature $-\kappa^2 \leq k(M) \leq 0$ and $diam(M) \leq D_0$.
%	\end{corollary}

	\begin{corollary} 
		\label{cor-intro-finiteness-of-manifolds}
		For every fixed $n, \kappa, D_0$, there exist only  finitely many compact $n$-dimensional Riemannian orbifolds $M$ with curvature $-\kappa^2 \leq k(M) \leq 0$ and $\textup{diam}(M) \leq D_0$, up to equivariant diffeomorphisms.\\
{\em (Two Riemannian orbifolds $M=\Gamma \backslash X$, $M'=\Gamma' \backslash X'$ are {\em equivariantly diffeomorphic} if there exists a diffeomorphism $X \rightarrow X'$ which is equivariant with respect to some  isomorphism $\varphi: \Gamma \rightarrow \Gamma'$.)}
	\end{corollary}
	
\noindent In particular,  for torsion-free orbifolds, this gives a new proof of the finiteness   of compact Riemannian  $n$-manifolds $M$ with   curvature $-\kappa^2 \leq k(M) \leq 0$ and  diam$(M) \leq D_0$,  modulo diffeomorphisms  
(this result was announced without proof by Gromov in  \cite{Gro78} and proved by Buyalo, up to homeomorphisms,  in \cite{buyalo}; the orbifold version was  proved by Fukaya \cite{Fuk86} in {\em strictly negative} curvature only).

% This result was  announced without proof by Gromov in \cite{Gro78}
%\footnote{In \cite{Gro78} Gromov refers to a forthcoming paper {\em``Nonpositively curved manifolds''}; but, as far as the authors are aware, {\color{red}a proof of this fact never appeared in literature.  Also, the book \cite{BGS13}, although containing many interesting   finiteness results (about the rank of the fundamental group and of the homology groups of $M$,  the number of ends of $M$ etc.) does not mention such a statement.}; } the orbifold version was  proved by Fukaya in {\em strictly negative} curvature, see \cite{Fuk86}.
%%\noindent This last result was  announced without proof by M. Gromov\footnote{In \cite{Gro78} Gromov  attributed the homotopy version of this result to H. Heintze \cite{Hei},   referring to a forthcoming paper {\em``Nonpositively curved manifolds''} for a proof in the differentiable setting. However, Heintze only obtained (via Cheeger's finiteness Theorem) the finiteness of $n$-manifolds $M$ with uniformly bounded diameter and curvature $-a^2 \leq k(M) \leq -b^2 <0$, proving a version of the celebrated Margulis'  Lemma in variable, pinched negative curvature. Also the book \cite{BGS13}, although containing several  finiteness results (about the rank of the fundamental group and of the homology groups of $M$, about the number of ends of $M$ etc.), as well as many other fundamental results in nonpositive curvature,  does not mention such a statement.}
\vspace{2mm}

Finiteness theorems in the spirit of Corollary \ref{cor-intro-finiteness-of-manifolds} have been proved in different contexts in literature, the archetype of all of them being   of course Weinstein's theorem for pinched, positively curved manifolds  \cite{Wei67} and Cheeger's finiteness theorems  in bounded sectional  curvature of variable sign  (see \cite{Che70}, \cite{Gro78b}, \cite{Pet84}, \cite{Yam85}).
%for complete proofs of Cheeger's theorem;  \cite{Fuk86} for an orbifold version.
%\cite{PT99} for simply connected  manifolds with finite second homotopy groups, etc.).
	 In  all these theorems,  
	 %such as   \cite{Che70} and \cite{Fuk86}, 
	the  finiteness is obtained by assuming  a positive lower bound on the injectivity radius, or  deducing such a lower bound from the combination of  geometric and topological assumptions (see for instance \cite{PT99} for simply connected  manifolds with finite second homotopy groups).
	  
A result similar to Theorem \ref{theo-intro-finiteness}  was recently proved in
%\cite{BCGS17}-
\cite{BCGS21}, where the  authors obtain the finiteness of torsionless groups $\Gamma$ acting faithfully, discretely and $D_0$-cocompactly on   $\delta_0$-hyperbolic  spaces    with  entropy $\leq H_0$. They achieve this by proving a positive, universal lower bound of the  {\em systole of $\Gamma$}, similar to the classical Heintze-Margulis' Lemma (holding for 
%of the pioneering work of  \cite{Hei}, which extends the classical Margulis' Lemma to 
manifolds with pinched, strictly negative curvature). 
Namely, for a  discrete isometry group $\Gamma$ of $X$ let us call, respectively,
$$\text{sys}(\Gamma, x) :=\inf_{g\in \Gamma \setminus \lbrace \text{id}\rbrace} d(x,gx) $$

\vspace{-3mm}

$$\text{sys}(\Gamma,X) 
	:= \inf_{x \in X} \text{sys}(\Gamma, x)$$
%= \inf_{x \in X} \inf_{g\in \Gamma \setminus \lbrace \text{id}\rbrace} d(x,gx)$$
the {\em systole of $\Gamma$ at $x$} and the (global) {\em systole of $\Gamma$} (notice that the systole of the fundamental group  $\Gamma=\pi_1(M)$ of a nonpositively curved manifold $M$, \linebreak acting on its Riemannian universal covering $X$, precisely equals twice the injectivity radius of $M$).
Then, in \cite{BCGS21}, the authors  prove that
$\text{sys}(\Gamma,X)$ is bounded below by a positive constant $ c_0 (\delta_0, H_0, D_0)$ only depending on the hyperbolicity constant, the entropy of $X$ and  on the diameter of $\Gamma \backslash X$.\linebreak
This bound  is, clearly, consequence of the Gromov hyperbolicity of the space $X$, which is a form of strictly negative curvature at macroscopic scale. \linebreak
In contrast, for the group actions in the classes considered in Theorem \ref{theo-intro-finiteness} and in the above corollaries,  the systole may well be arbitrarily small, since the curvature is only assumed to be non-positive. The main difficulty in Theorem  \ref{theo-intro-finiteness} and in the corollaries  above precisely boils down  in  understanding what happens when the systole or the injectivity radius tend to zero. 
%	If  $\Gamma$ is a cocompact, discrete   group of isometries of a  CAT$(0)$-space  $X$ which, moreover, is  {\em hyper\-bolic} in the sense of Gromov,  a  packing condition on $X$ implies a universal lower bound  of the minimal displacement of every torsion-free   isometry of  $\Gamma$, which only depends on the hyperbolicity constant $\delta$, on the packing constants $(P_0,r_0)$, and on an upper bound $D$ of the diameter of  the quotient $\Gamma \backslash X$,  as recently proved in \cite{BCGS21} {\color{red}and \cite{CavS20bis}}:
%	$$\text{sys}^\diamond(\Gamma, X) := \inf_{x \in X}\inf_{g \in \Gamma^\diamond} d(x,gx) \ge \ell_0 (\delta,  P_0, r_0,   D)>0$$
%	where the quantity $\textup{sys}^\diamond(\Gamma,X)$ above is called  the {\em free systole} of the action. 

	%This means that a family a compact, geodesically complete locally CAT$(\kappa)$-spaces with $\kappa<0$, which are $(P_0,r_0$-packed and have uniformily bounded diameter cannot collapse, since their injectivity radius is bounded below by a universal positive constant.

%	This implies, in particular, that a family of compact,  $(P_0,r_0)$-packed,
%	locally CAT$(\kappa)$-spaces with strictly negative $\kappa$, 
% 	with uniformily bounded diameter,  cannot collapse.
%	In contrast,  we are interested in  quotients of  spaces which are only locally   CAT$(0)$,  and  their  free systole is not universally bounded away from zero, so  they  can collapse.
\vspace{1mm}

In fact, the problem of {\em collapsing} will be of primary interest in this work.
	Recall that a Riemannian  manifold $M$ is called $\varepsilon$-collapsed if the  injectivity radius  is smaller than  $\varepsilon$ at every point. 
	The theory of collapsing for Riemannian manifolds with bounded sectional curvature was developed by  Cheeger, Gromov and Fukaya:  for  a differentiable manifold $M$, the existence of a Riemannian metric with bounded sectional curvature sufficiently collapsed imposes strong restrictions to its topology. Namely, there  exists an $\varepsilon_0(n)>0$ such that if a Riemannian $n$-manifold with $|K_M |\leq 1$ \linebreak is $\varepsilon$-collapsed with $\varepsilon < \varepsilon_0(n)$,  then $M$ admits a so-called $F$-structure of positive rank, cp. \cite{CG86}-\cite{CG90} and \cite{CFG92} (see also  the works of Fukaya \cite{Fuk87}-\cite{Fuk88} for collapsible manifolds with uniformly bounded diameter).
	%and the foundational works of Cheeger-Colding \cite{CC97}, \cite{CC00}, \cite{CC00b} and  Cheeger-Colding-Tian \cite{CCT02} for the structure of singularities of collapsed limits of Riemannian manifolds  assuming only Ricci curvature bounded below).
	
More specifically about  nonpositively curved geometry, Buyalo \cite{Buy1}-\cite{Buy2} first, in dimension smaller than $5$, and Cao-Cheeger-Rong \cite{CCR01} later, in any dimension,  studied the possibility of collapsing compact,  $n$-dimensional manifolds $M$ with bounded sectional curvature  $-\kappa^2 \le K_M \le 0$.
	They proved that either the injectivity radius at some point %$p \in M$ 
	is bounded below by a universal positive constant $i_0(n,\kappa) > 0$, or 
	%the action of $\Gamma$ is $\eta$-thin and 
	$M$ admits  a so-called {\em abelian local splitting structure}: this is, roughly speaking, a decomposition of the universal covering $X$ into a union of minimal sets of hyperbolic isometries with the additional property that if two minimal sets intersect then the corresponding isometries commute. 
	%We partially generalized this result to CAT$_0$-spaces in \cite{CS22}.
%	
		A prototypical example of collapsing with bounded, non-positive curvature is the following, which might be useful to have in mind for the sequel (cp. \cite{Gro78}, Section 5 and \cite{Buy81}, Section 4): \linebreak consider two copies $\Sigma_1, \Sigma_2$ of the same hyperbolic surface with connected, geodesic boundary of length $\frac{1}{n}$, then take the products $\Sigma'_i= \Sigma_i \times S^1$ with a circle of length $\frac{1}{n}$, and glue $\Sigma'_1$ to $\Sigma'_2$ by means of an isometry  $\varphi$ of the boundaries $\partial \Sigma'_i= \partial \Sigma_i \times S^1$ which interchanges the circles $\partial \Sigma_i$ with  $S^1$.   This yields a nonpositively curved $3$-manifold $M_n$ (a {\em graph manifold}) with sectional curvature $-1 \leq K_{M_n} \leq 0$, whose injectivity radius at every point is arbitrarily small provided that $n \gg 0$. 
%Notice however that the manifolds $M_n$ do not have uniformly bounded diameter: their limit for $n \rightarrow 0$ in the pointed, Gromov-Hausdorff topology  is either a hyperbolic surface with a cusp, or the real line ${\mathbb R}$, according to the choice of the base points $x_n \in M_n$. 
\vspace{2mm}

 Coming back to our metric setting,  where we also allow groups with torsion and isometries with fixed points, it is useful to 
 %a space $X$ can admit subtle thin actions: for instance, it may happen that every point is  fixed by some elliptic isometry of  $\Gamma$ (see for instance the Example 1.4 in \cite{CS22}), and thus   $\text{dias}(\Gamma, X)=0$.	We call {\em singular} such degenerate actions. A discrete action is clearly non-singular when $\Gamma$ is torsionless, or when $X$ is a manifold (or even	a homology manifold, see Proposition 2.9 of \cite{CS22}.
distinguish between systole and free systole: we define the {\em free systole of $\Gamma$ at $x$} and the (global)  {\em free systole of $\Gamma$}  respectively as
\vspace{-4mm}

$$\text{sys}^\diamond(\Gamma,x) = \inf_{g \in \Gamma^\diamond} d(x,gx)$$

\vspace{-3mm}

$$\text{sys}^\diamond(\Gamma,X) = \inf_{x \in X} \text{sys}^\diamond(\Gamma,x)$$
where $\Gamma^\diamond$ denotes the subset of torsion-free elements of $\Gamma$.\\
The first non-trivial finding for non-singular actions is that, when assuming a bound on the diameter,  then 
%a positive lower bound for the global free systole is quantitatively equivalent to a positive lower bound  for the systole at some point (see Theorem \ref{prop-vol-sys-dias})
the smallness  of the  systole  at every point is quantitatively equivalent to the smallness  of the global free systole (see Theorem \ref{prop-vol-sys-dias}). 
Therefore, we will say that a $D_0$-cocompact action of a group $\Gamma$ on a CAT$(0)$-space $X$ (or, equivalently, the quotient space $M=\Gamma \backslash X$) is  {\em $\varepsilon$-collapsed} if $\textup{sys}^\diamond(\Gamma,X) \leq \varepsilon$. 

\vspace{1mm}
%	The condition that $X$ admits a discrete, $D_0$-cocompact action which is  $\varepsilon$-collapsed  for sufficiently small $\varepsilon$   imposes a strong  restriction on $X$. Namely, the following theorem shows that $X$ necessarily splits a non-trivial Euclidean factor:
The following theorem, which is the key to our finiteness theorems, 
shows that if $X$ admits a discrete, $D_0$-cocompact action which is  $\varepsilon$-collapsed  for sufficiently small $\varepsilon$,  
then  $X$ necessarily splits a non-trivial Euclidean factor:

\begin{theorem}[Splitting of an Euclidean factor under $\varepsilon$-collapsed actions]
% \label{theo-space-splitting} 
	\label{theo-intro-splitting} ${}$\\
Let $P_0,r_0,D_0>0$. There exists $\sigma_0 = \sigma_0(P_0,r_0,D_0) > 0$ such that  if $X$ is a  proper, geodesically complete, $(P_0,r_0)$-packed, $\textup{CAT}(0)$-space     admitting a discrete, $D_0$-cocompact group of isometries   $\Gamma$ with
%   $\textup{diam}(\Gamma\backslash X) \leq D_0$ and 
$\textup{sys}^\diamond(\Gamma,X) \leq \sigma_0$ then $X$ splits isometrically as $Y \times \mathbb{R}^k$, with $k\geq 1$.
\end{theorem}

%	\begin{theorem}[Splitting of an Euclidean factor under $\varepsilon$-collapsed actions]
%		\label{theo-intro-splitting}
%		Let $P_0,r_0,D_0 \!> 0$. There exists $\sigma_0 = \sigma_0(P_0,r_0,D_0) \!> 0$ such that, for any $(X, \Gamma)$ in CAT$_0 (P_0, r_0, D_0)$, if $\textup{sys}^\diamond(\Gamma,X) \leq \sigma_0$ then $X$ splits isometrically as $Y\times \mathbb{R}^k$ for some $k\geq 1$.
%	\end{theorem}
 	
\noindent Notice that, as the prototypical example  above shows already for manifolds, the splitting of $X$ does not hold without an a priori bound on the   diameter. 
 \noindent  As proved by Buyalo in \cite{buyalo}, when $M= \!\Gamma \!\backslash X$ is a manifold  with  pinched, nonpositive sectional curvature   $ -1 \leq  k_M    \leq\! 0$  and diameter bounded by $ D$,   
 %if we know a bound {\em a priori}  $\sigma(g) \leq N$ on the order of all torsion elements  $g \in \Gamma$,   
 then there exists a positive constant $\varepsilon(n,D)$  such that the condition of being $\varepsilon$-collapsed for $\varepsilon < \varepsilon(n,D)$ implies  the existence of a   normal, free abelian subgroup $A$ of rank $ k \geq 1$; then,
 $\Gamma$ virtually splits $A$, and  the existence of a non-trivial Euclidean factor for $X$  can  be deduced from  classical splitting theorems   for  non-positively curved manifolds   	(see   \cite{Ebe88}, or \cite[Prop.6.23]{BH09}). \linebreak
 The proof of  Buyalo uses a  stability property of hyperbolic isometries based on   the fact that two isometries which coincide on an open subset are equal,  a fact which is drastically false   in our metric context (see the discussion in \cite[Example 1.2]{CS22}).
 Also, notice that in Theorem \ref{theo-intro-splitting} we do not assume any  bound on the order of torsion elements  $g \in \Gamma$ (as opposite to  \cite{Fuk86}):
% (this would yield  a   normal, free abelian subgroup $A$ of rank  $k\geq 1$ by Cheeger-Rong's arguments\footnote{Actually, in \cite{CR95}, Corollary 0.8, it is erroneously asserted that, if  $M$ is non-positively curved and $\varepsilon$-collapsed with $\varepsilon < \varepsilon(n,D)$, then $\Gamma$ contains a normal, free abelian subgroup of rank $k\geq 2$. This  is clearly false: for instance, for any compact hyperbolic surface $\Sigma$,  the fundamental group of $\Sigma \times S^1$ does not have such a subgroup. However, the existence of a pure injective $F$-structure proved in their Theorem 0.5 implies the existence of a normal, free abelian subgroup of rank $k\geq 1$.} in \cite{CR95}):
 actually, bounding the order of the torsion elements of $\Gamma$  is one of the main conclusions of this work, see Corollary \ref{cor-intro-order-elements} below.\\
Our proof is  more inspired to \cite{CM09b}-\cite{CM19}: we do not prove the existence of a normal free abelian subgroup, rather we find a free abelian, {\em commensurated} subgroup $A'$ of $\Gamma$, from which we  construct a $\Gamma$-invariant closed convex subset $X_0 \subset X$ which splits as $Y\times {\mathbb R}^k$, and then we use  the minimality of the  action  to deduce that $X_0=X$ (see Section \ref{sec-splitting} for the proof).  Actually, in our metric setting it can happen that $\Gamma$ does not have any non-trivial, normal, free abelian subgroup  at all (see \cite{CS23}). 

%Also, our proof is conceptually very different from the approach in *****\cite{CR95},*****} and more inspired to \cite{CM09a}-\cite{CM19}: we do not prove the existence of a normal free abelian subgroup, rather we find a free abelian, {\em commensurated} subgroup $A'$ of $\Gamma$, from which we  construct a $\Gamma$-invariant closed convex subset $X_0 \subset X$ which splits as $Y\times {\mathbb R}^k$, and then we use  the minimality of the  action  to deduce that $X_0=X$ (see Section \ref{sec-splitting} for the proof). Also notice that,   in Theorem \ref{theo-intro-splitting}, besides the {\color{blue}wider generality,}  we do not assume any bound on the torsion of the group $\Gamma$;  bounding the order of the torsion elements of $\Gamma$ actually is  one of the  key-findings of this work, see Corollary \ref{cor-intro-order-elements} below. 
	 \vspace{2mm}	

%Theorem \ref{theo-intro-splitting} is the starting point for our finiteness theorem. 
The final step  for Theorem  \ref{theo-intro-finiteness} is realizing that the $\varepsilon$-collapsing of an action of $\Gamma$ on $X$ can occur on different subspaces of $X$ at different scales; a careful analysis of this phenomenon allows us to {\em renormalize}, in a precise sense, the metric of $X$, keeping the diameter of the quotient $\Gamma \backslash X$ bounded, and obtaining  the following  result, similar to Buyalo's \cite[Theorem 1.1]{buyalo}, which we believe is of independent interest (see Section \ref{finiteness} and Remark \ref{rmk-splitting} for a more precise statement):

\begin{theorem}[Renormalization]
	\label{theo-bound-systole}${}$\\
%	Given $P_0,r_0,D_0>0$, there exist  two positive constants $s_0 = s_0(P_0,r_0,D_0)$ and $\Delta_0 = \Delta_0(P_0,D_0)$ such that the following holds. \linebreak
Given $P_0,r_0,D_0$, there exist $s_0 = s_0(P_0,r_0,D_0)>0$ and $\Delta_0 = \Delta_0(P_0,D_0)$ such that the following holds.
Let $\Gamma$ be a discrete and $D_0$-cocompact isometry group of 
a proper, geodesically complete, $(P_0,r_0)$-packed, $\textup{CAT}(0)$-space $X$:  \linebreak 
then,	$\Gamma$ admits also a faithful, discrete, $\Delta_0$-cocompact action by isometries on a $\textup{CAT}(0)$-space $X'$   isometric to $X$,  such that 
	 $\textup{sys}^\diamond(\Gamma, X') \geq s_0$.\\
Moreover, the action of $\Gamma$ on $X$ is nonsingular if and only if the action on $X'$ is nonsingular.
\end{theorem}

%\begin{theorem}[Renormalization]
%		\label{theo-intro-splitting-strong}${}$\\
%		Given $P_0,r_0,D_0 > 0$, there exist two positive constants $s_0 = s_0(P_0,r_0,D_0)$ and $\Delta_0 = \Delta_0(P_0,r_0,D_0)$ such  that  for any 
%		$(X, \Gamma)$ in CAT$_0 (P_0, r_0, D_0)$, the group $\Gamma$ admits a faithful, discrete, $\Delta_0$-cocompact action on another proper, geodesically complete, $(P_0,r_0)$-packed, $\textup{$-space $X'$ with $\textup{sys}^\diamond(\Gamma,X') \geq s_0$.
%\end{theorem}
 
\noindent This theorem allows us to deduce the finiteness of  nonsinngular groups $\Gamma$ in the class CAT$_0 (P_0, r_0, D_0)$, using Serre's classical  presentation of groups acting on simply connected spaces, and  to obtain Corollary \ref{cor-intro-finiteness-of-manifolds} from Cheeger's and Fukaya's  finiteness theorems.  
\vspace{2mm}

%\begin{obs*}
\noindent {\em Remark.}
Explicitly, we may take $\Delta_0 = 2^{n_0}(D_0 +\sqrt{n_0})$, where $n_0=P_0/2$ bounds the dimension of every $X$ in CAT$_0 (P_0, r_0, D_0)$ (see Proposition 	\ref{prop-packing}).
Also the constant $s_0$ could be made explicit in terms of $P_0,R_0,D_0$ and  of the Margulis' constant $\varepsilon_0 =\varepsilon_0 (P_0,r_0)$,  
following the proof of Theorem \ref{theo-bound-systole} and using the function $\sigma_0(P_0,r_0,D_0)$ appearing in  Theorem \ref{theo-intro-splitting}  (itself explicitable in terms of $P_0,R_0,D_0$ and  of the universal  bound of the index of the lattice of translations ${\mathcal L}(G)$ in any crystallographic group $G$ of dimension $\leq n_0$).  The only quantity which is not explicit is the Margulis constant $\varepsilon_0(P_0,r_0)$ provided by \cite{BGT11}.
%\end{obs*}
\vspace{2mm}

Another immediate  (although a-priori highly non-trivial) consequence of Theorem \ref{theo-intro-finiteness}  is that there exists a uniform bound on the order of the finite subgroups of the groups $\Gamma$ under consideration. We stress that the existence of such a bound is new also for isometry groups $\Gamma$ of Hadamard manifolds; for instance, this is a key-assumption in the theory of convergence of Riemannian orbifolds of Fukaya. A similar bound  is proved in  \cite{Fuk86}
%, Lemma 9.2, 
under the additional assumption of a lower bound on the volume of $M=\Gamma \backslash X$.
More explicitly,  there exists  a constant $b_0 = b_0(P_0,r_0) > 0$ (whose geometric meaning is explained in Proposition \ref{prop-decomposition-CS})
% see also  \cite[Theorem A \& Remark 3.5]{CS22}
 such that the  following holds:

	\begin{corollary}[Universal bound of the order of finite subgroups]
		\label{cor-intro-order-elements}
		${}$\\
		Let $P_0,r_0,D_0 \!> 0$. Then for every  nonsingular 
		$\Gamma \curvearrowright X$ in \textup{CAT}$_0 (P_0, r_0, D_0)$, 
		 every finite subgroup  $F <\Gamma$ has order
		 % $\leq N_0(P_0,r_0,D_0)$. In particular the order of every elliptic isometry is bounded above by 
		\vspace{-4mm} 
		 
$$ |F| \leq \frac{V(\Delta_0)}{v\left(\frac12 b_0s_0\right)}$$
(where   $v,V$ are the universal functions appearing in  \eqref{eq-intro-volume-estimate}, and $\Delta_0,s_0$ are  the  constants of Theorem \ref{theo-bound-systole}).
	\end{corollary}

\small {\em
\noindent {\sc Acknowledgments.} The authors thank P.E.Caprace, S.Gallot and A.Lytchak for many interesting discussions during the preparation of this paper, and D.Semola for pointing us to interesting references.}
\normalsize
 
\vspace{2mm}

\small
 \noindent{\sc Notation.} Throughout the paper, we will adopt the following convention:  a letter with the subscript $_0$ will always denote  a fixed constant or  a universal function depending only on the parameters $P_0,r_0$ and, possibly, on $D_0$. For instance, in the above theorems,  we used the constants $\sigma_0, s_0, \Delta_0, b_0$ with this meaning.
\normalsize

\section{Preliminaries on CAT$(0)$-spaces}
\label{sec-CAT}

We fix here some notation and recall some facts about CAT$(0)$-spaces. \\
Throughout the paper $X$ will be a {\em proper} metric space with distance $d$.
The open (resp. closed) ball in $X$ of radius $r$, centered at $x$, will be denoted by $B_X(x,r)$ (resp.  $\overline{B}_X(x,r)$); we will often drop the subscript $X$ when the space  is clear from the context.\\
A {\em geodesic} in a metric space $X$ is an isometry $c\colon [a,b] \to X$, where $[a,b]$ is an interval of $\mathbb{R}$. The {\em endpoints} of the geodesic $c$ are the points $c(a)$ and $c(b)$; a geodesic with endpoints $x,y\in X$ is also denoted by $[x,y]$. A {\em geodesic ray} is an isometry $c\colon [0,+\infty) \to X$ and a geodesic line is an isometry $c\colon \mathbb{R} \to X$. A metric space $X$ is called   {\em geodesic}  if for every two points $x,y \in X$ there is a geodesic with endpoints   $x$ and $y$. 

\noindent A metric space $X$ is called CAT$(0)$ if it is geodesic and every geodesic triangle $\Delta(x,y,z)$  is thinner than its Euclidean comparison triangle   $\overline{\Delta} (\bar{x},\bar{y},\bar{z})$: that is,  for any couple of points $p\in [x,y]$ and $q\in [x,z]$ we have $d(p,q)\leq d(\bar{p},\bar{q})$ where $\bar{p},\bar{q}$ are the corresponding points in $\overline{\Delta} (\bar{x},\bar{y},\bar{z})$ (see   for instance \cite{BH09} for the basics of CAT(0)-geometry).
As a consequence, every CAT$(0)$-space is {\em uniquely geodesic}: for every two points $x,y$ there exists a unique geodesic with endpoints $x$ and $y$.

\noindent A CAT$(0)$-metric space $X$ is {\em geodesically complete} if evvery geodesic $c\colon [a,b] \to X$ can be extended to a geodesic line. For instance, if a CAT$(0)$-space is a homology manifold then it is always geodesically complete, see \cite[Proposition II.5.12]{BH09}).

\noindent The  {\em boundary at infinity} of a CAT$(0)$-space $X$ (that is, the set of equivalence classes of geodesic rays, modulo the relation of being asymptotic), endowed with the Tits distance,  will be denoted by $\partial X$, see \cite[Chapter II.9]{BH09}.

\noindent A subset $C$ of $X$ is said to be {\em convex} if for all $x,y\in C$ the geodesic $[x,y]$ is contained in $C$. Given a subset $Y\subseteq X$ we denote by $\text{Conv}(Y)$ the {\em convex closure} of $Y$, that is the smallest closed convex subset containing $Y$.\linebreak
If $C$ is a convex subset of a CAT$(0)$-space $X$ then it is itself CAT$(0)$, and its boundary at infinity $ \partial C$ naturally and isometrically embeds in $\partial X$.
\vspace{1mm}

%\subsection{Canonical measure and dimension ?}	
 We will denote by HD$(X)$ and TD$(X)$ the Hausdorff and the topological dimension of a metric space $X$, respectively.  
 By \cite{LN19} we know that if $X$ is a  proper and geodesically complete CAT$(0)$-space then every point $x\in X$ has a well defined integer dimension in the following sense: there exists $n_x\in \mathbb{N}$ such that every small enough ball around $x$ has Hausdorff dimension equal to $n_x$. This defines a {\em stratification of $X$} into pieces of different integer dimensions: namely, if $X^k$ denotes the subset of points of $X$ with dimension $k$, then 
 \vspace{-5mm}
 
 $$X= \bigcup_{k\in \mathbb{N}} X^k.$$
The {\em dimension} of $X$ is the supremum of the dimensions of its points:  it coincides with the {\em topological dimension} of $X$, cp. \cite[Theorem 1.1]{LN19}.

\noindent Calling $\mathcal{H}^k$ the $k$-dimensional Hausdorff measure, the formula
 \vspace{-3mm}

$$\mu_X := \sum_{k\in \mathbb{N}} \mathcal{H}^k \llcorner X^k$$  defines a {\em canonical measure} on $X$ which is locally positive and locally finite. 

\subsection{Discrete isometry groups}
\label{subsection-isometries}
${}$

%Let $X$ be a proper metric space. 
%\noindent We denote by $\text{Isom}(X)$ the group of isometries of $X$, endowed with the compact-open topology: it is a topological,  locally compact  group. \\
\noindent Let $\text{Isom}(X)$ be the group of isometries of $X$, endowed with the compact-open topology: as $X$ is proper, it is a topological,  locally compact  group. \\
The {\em translation length} of $g\in \text{Isom}(X)$ is by definition  $\ell(g) := \inf_{x\in X}d(x,gx).$ 
When the infimum is realized, the isometry $g$ is called {\em elliptic} if $\ell(g) = 0$ and {\em hyperbolic} otherwise. The {\em minimal set of $g$}, $\text{Min}(g)$, is defined as the subset of points of $X$ where $g$ realizes its translation length; notice that if $g$ is elliptic then $\text{Min}(g)$ is the subset of points fixed by $g$. An isometry is called \emph{semisimple} if it is either elliptic or hyperbolic; a subgroup $\Gamma$ of Isom$(X)$ is  called \emph{semisimple} if all of its elements are semisimple.\\
Let $\Gamma$ be a subgroup of Isom$(X)$. For $x\in X$ and $r\geq 0$ we set 
\begin{equation}	
	\label{defsigma} 	
	\overline{\Sigma}_r(x,X) := \lbrace g\in \Gamma \text{ s.t. } d(x,gx) \leq r\rbrace
\end{equation}

\vspace{-7mm}

\begin{equation}	
	\label{defgamma} 	
	\overline{\Gamma}_r(x,X) := \langle \overline{\Sigma}_r(x, X) \rangle
\end{equation}
When the context is clear we will simply write $\overline{\Sigma}_r(x)$ and $\overline{\Gamma}_r(x)$.\\
The subgroup $\Gamma$ is \emph{discrete} if it is  discrete as a subset of \textup{Isom}(X) (with respect to the compact-open topology). Since $X$ is assumed to be proper and $\Gamma$ acts by isometries, this condition is the same as asking that the orbit $ \Gamma x$ is discrete and $\text{Stab}_\Gamma (x)$ is finite for some (or, equivalently, for all) $x\in X$. \linebreak This is in turn equivalent to asking that the sets $\overline{\Sigma}_r(x)$ are finite for all $x\in X$ and all $r\geq 0$.\\
%\noindent If $X$ is any proper metric space we denote by Isom$(X)$ its group of isometries, endowed with the uniform convergence on compact subsets of $X$. A subgroup $\Gamma$ of Isom$(X)$ is called {\em discrete} if the following (equivalent) conditions (see \cite{BCGS17}) hold:
%\vspace{-2mm}
%\begin{itemize}
%	\item[(a)] $\Gamma$ is discrete as a subspace of Isom$(X)$; \vspace{-2mm}
%	\item[(b)] $\forall x\in X$ and $R\geq 0$ the set $\Sigma_R(x) = \lbrace g \in \Gamma  \hspace{1mm} | \hspace{1mm}  g  x\in \overline{B}(x,R)\rbrace$ is finite.
%\end{itemize}  
When dealing with  isometry groups $\Gamma$ of CAT$(0)$-spaces with torsion, a   difficulty is that  there may exist nontrivial % different isometries which coincide on open sets. This is related to the existence of 
elliptic isometries which act as the identity on open sets. 
Following \cite[Chapter 11]{dlHG90}, a subgroup   $\Gamma$ of Isom$(X)$ will be called  {\em rigid} (or \emph{slim}, with the terminology used in \cite{CS22}) if for all $g \in \Gamma$ the subset $\text{Fix}(g)$ has empty interior. Every torsion-free group is trivially rigid, as well as any discrete group acting on a CAT$(0)$-homology manifold, as proved in \cite[Lemma 2.1]{CS22}.\\
A {\em \textup{CAT}$(0)$-orbispace} 
(in the sense of \cite{Fuk86})
%(resp.  CAT$(0)$-homology orbifold) 
is the quotient $M=\Gamma \backslash X$ \linebreak 
of a (proper, geodesically complete) \textup{CAT}$(0)$-space $X$ by a discrete isometry group $\Gamma$.
One might define the notion of orbispace $M$ in terms of an orbifold atlas, that is  a collection of uniformizing charts $\pi_i: V_i \rightarrow U_i \subset M$ covering $M$ (where $V_i$ is a locally compact space endowed with the action of a finite group $\Gamma_i$ such that $\pi_i$ induces a homeomorphism $\Gamma_i \backslash V_i \simeq U_i$) and a pseudogroup of local homeomorphisms given by changing charts. 
This is for instance the approach of Haefliger in \cite[Ch.11]{dlHG90}, where {\em rigid orbispaces} are defined; {\em rigid} here means that the actions of the  groups $\Gamma_i$ on $V_i$ are all supposed to be rigid.
Every quotient of a CAT$(0)$-space $X$ by a discrete and rigid group $\Gamma$ has a structure of {\em rigid orbispace}; reciprocally, every rigid orbispace $M$ {\em with nonpositive curvature} (that is, such that the domains $V_i$ of the uniformizing charts are locally CAT$(0)$-spaces) is {\em developable}, which means that  $M=\Gamma \backslash X$, for some rigid, discrete isometry group $\Gamma$ acting on a suitable CAT$(0)$-space $X$, see \cite[Ch.11, Th\'eor\`eme 8]{dlHG90}.\\
A group $\Gamma < \textup{Isom}(X)$ is said to be {\em cocompact} if the quotient metric space $\Gamma \backslash X$ is compact; in this case, we call \emph{codiameter} of $\Gamma$ the diameter of the quotient, and we will say that $\Gamma$ is $D_0$-cocompact if it has codiameter at most $D_0$. \linebreak
Notice that the codiameter of $\Gamma$ coincides with
	\begin{equation}
		\label{eq-def-codiameter}
		\inf \lbrace r>0 \text{ s.t. } \Gamma \cdot \overline{B}(x,r) = X \,\,\,\,\forall x\in X\rbrace.
	\end{equation}	
It is well-known, and we will consistently use it in the paper, that if $X$ is geodesic and $\Gamma < \textup{Isom}(X)$ is discrete and $D_0$-cocompact, then for all $D \geq D_0$  {\em the subset  $\overline{\Sigma}_{2D}(x)$ is a  generating set for $\Gamma$,} that is $\overline{\Gamma}_{2D}(x)=\Gamma$, for every $x\in X$; %(see \cite{Ser80}, \cite{Gro78b}). 
we call this a {\em $2D$-short generating set  of $\Gamma$  at  $x$.}\\
Moreover, if $X$ is simply connected, then $\Gamma$ admits a finite presentation as 
$$\Gamma = \langle  \overline{\Sigma}_{2D}(x) \hspace{1mm} |  \hspace{1mm}{\mathcal R}_{2D}(x)  \rangle$$
 where ${\mathcal R}_{2D}(x)$ is a subset of words of length $3$ on $\overline{\Sigma}_{2D}(x)$, see \cite[App., Ch.3]{Ser80}. For later use we  also record the following general fact.

	\begin{lemma}
		\label{lemma-cocompact-finite-index}
		Let $X$ be a geodesic metric space and let $\Gamma < \textup{Isom}(X)$ be a discrete, $D_0$-cocompact group. 
		If  $\Gamma'$ is normal subgroup of  $\Gamma$ with index $[\Gamma : \Gamma'] \leq J$, then $\Gamma'$ is at most $2D_0(J+1)$-cocompact.
	\end{lemma}
	\begin{proof}
	As $\Gamma'$ is discrete, the space  $Y=\Gamma'\backslash X $ is geodesic and the group $\Gamma' \backslash \Gamma$ acts by isometries on it with codiameter at most $D_0$.
	Take arbitrary $y,y' \in Y$ and connect them by a geodesic $c$.
%\colon [0,d(y,y')+\varepsilon]$ parametrized by arc-length. 
Let $t_k= k\cdot (2D_0+\varepsilon)$, as far as $c(t_k)$ is defined for integers $k$, and let $g_k\in \Gamma' \backslash \Gamma$ be such that $d(c(t_k), g_ky) \leq D_0$.  \linebreak
By construction, the points $\lbrace g_ky \rbrace$ are all distinct since they are all $\varepsilon$-separated, so the $g_k$'s are distinct. Since $\Gamma' \backslash \Gamma$ has cardinality at most $J$, we conclude that $d(y,y') \leq (J+1)\cdot(2D_0 + \varepsilon)$. By the arbitrariness of $\varepsilon$, $y$ and $y'$ we deduce that the diameter of $Y$ is at most $2D_0(J+1)$.	
\end{proof}

In general the full isometry group Isom$(X)$ of a CAT$(0)$-space is not a Lie group, for instance in the case of regular trees.
%In this section we recall the main properties of a geodesically complete CAT$(0)$-space admitting a discrete, cocompact group of isometries proved in \cite{CM09}. So let us fix a complete, geodesically complete, CAT$(0)$-space $X$ and a discrete, cocompact group of isometries $\Gamma$ of $X$.
When a CAT$(0)$-space $X$ admits a {\em cocompact}, discrete group of isometries $\Gamma$ then Isom$(X)$ is known to have more structure, as proved by P.-E.Caprace and N.Monod.

\begin{prop}[\textup{\cite[Thm.1.6 \& Add.1.8]{CM09b}, \cite[Cor.3.12]{CM09a}}]
	\label{prop-CM-decomposition}${}$
	
\noindent	Let $X$ be a proper, geodesically complete, $\textup{CAT}(0)$-space, admitting a discrete, cocompact group of isometries.
% $\Gamma < \textup{Isom}(X)$ discrete and cocompact. 
Then $X$ splits isometrically as $M \times \mathbb{R}^n \times N$, where $M$ is a symmetric space of noncompact type and ${\mathcal D}:=\textup{Isom}(N)$ is totally disconnected. Moreover 
	$$\textup{Isom}(X) \cong {\mathcal S}   \times {\mathcal E}_n \times {\mathcal D}$$
where ${\mathcal S}$ is a semi-simple Lie group with trivial center and without compact factors and ${\mathcal E}_n \cong \textup{Isom}(\mathbb{R}^n)$.	
\end{prop}

\subsection{The packing condition and Margulis' Lemma}
\label{subsection-packing-margulis}
${}$

\noindent Let $X$ be a metric space and $r>0$.
A subset $Y$ of $X$ is called {\em $r$-separated} if $d(y,y') > r$ for all $y,y'\in Y$. Given $x\in X$ and $0<r\leq R$ we denote by Pack$(\overline{B}(x,R), r)$ the maximal cardinality of a $2r$-separated subset of $\overline{B}(x,R)$. Moreover we denote by Pack$(R,r)$ the supremum of Pack$(\overline{B}(x,R), r)$ among all points of $X$. 
Given $P_0,r_0 > 0$ we say that $X$ is {\em $P_0$-packed at scale $r_0$} (or $(P_0,r_0)$-packed, for short) if Pack$(3r_0,r_0) \leq P_0$.  We will simply say that {\em $X$ is packed} if it is $P_0$-packed at scale $r_0$ for some $P_0,r_0>0$. \\
The packing condition should be thought as a metric, weak replacement  of 
%a lower bound on the Ricci curvature: 
a Ricci curvature lower bound: for more details and examples see \cite{CavS20}. 	\linebreak
Actually, by Bishop-Gromov's Theorem, for a $n$-dimensional Riemannian manifold a lower bound on the Ricci curvature  $\text{Ric}_X \geq -(n-1)\kappa^2$ implies a uniform estimate of the packing function at any fixed scale $r_0$, that is 
\begin{equation}\label{eq-bishop}
\text{Pack}(3r_0,r_0) \leq \frac{v_{{\mathbb H}^n_{\kappa}}(3r_0)}{v_{{\mathbb H}^n_{\kappa}}(r_0)}
\end{equation}
where $v_{{\mathbb H}^n_{\kappa}}(r)$ is the volume of a ball of radius $r$ in the $n$-dimensional space form with constant curvature $-\kappa^2$.\\
  Also remark that every metric space admitting a cocompact action is packed (for some $P_0, r_0$), see the proof of \cite[Lemma 5.4]{Cav21ter}. 
\vspace{1mm}

\noindent The packing condition  has many interesting geometric consequences for complete, geodesically complete CAT$(0)$-spaces, as showed in \cite{CavS20bis}, \cite{Cav21bis} and  \cite{Cav21}. The first one is a uniform estimate of the measure of balls of any radius $R$ and an upper bound on the dimension, which we summarize here.

\begin{prop}[\textup{\cite[Thms. 3.1, 4.2, 4.9]{CavS20}, \cite[Lemma 3.3]{Cav21}}]
	\label{prop-packing}
	${}$
	
\noindent	Let $X$ be a complete, geodesically complete, $(P_0,r_0)$-packed, $\textup{CAT}(0)$-space. Then $X$ is proper and 
	\begin{itemize}%[leftmargin=10mm]	 
		\item[(i)] $\textup{Pack}(R,r) \leq P_0(1+P_0)^{\frac{R}{\min\lbrace r,r_0\rbrace} - 1}$ for all $0 <r\leq R$;	 
		\item[(ii)] the dimension of $X$ is at most $n_0 := P_0/2$;
		\item[(iii)] there exist functions $v,V \colon (0,+\infty) \to (0,+\infty)$ depending only on $P_0,r_0$ such that for all $x\in X$ and $R> 0$ we have
		\begin{equation}
			\label{eq-volume-estimate}
			v(R) \leq \mu_X(\overline{B}(x,R)) \leq V(R);
		\end{equation}
		 \item[(iv)] The entropy of $X$ is bounded above in terms of $P_0$ and $r_0$, namely
			$$\textup{Ent}(X) := \limsup_{R \to +\infty} \frac{1}{R}\log \textup{Pack}(R,1) \leq \frac{\log (1+P_0)}{r_0}.$$
	\end{itemize}
\end{prop}

		%We fix $x\in X$. By Proposition \ref{prop-generating-set-cocompact} the set $S=\Sigma_{2D}(x)$ generates $\Gamma$. By definition of codiameter for every $y\in X$ there exists $g\in \Gamma$ such that $d(y,gx) \leq D$. Moreover following the proof of Lemma \ref{lemma-generating-finite-index} for every $g\in \Gamma$ there exists $g'\in \Gamma'$ and $h \in S^{2J} \subseteq \Sigma_{4JD}$ such that $g=g'h$, i.e. $g' = gh^{-1}$. Therefore 
		%$$d(y,g'x) = d(y, gh^{-1}x) \leq d(y,gx) + d(gx, gh^{-1}x) \leq D + 4JD = (4J+1)D.$$
		%This shows that $\Gamma'$ is cocompact with codiameter at most $(4J+1)D$.

\noindent In particular,  for a geodesically complete, CAT$(0)$ space $X$ which is $(P_0, r_0)$-packed, the assumptions {\em complete} and {\em proper} are interchangeable.

\noindent Also, property (i) shows that, for complete and geodesically complete CAT(0)-spaces, a packing condition at some scale $r_0$ yields an explicit, uniform control of the packing function at any other  scale $r$: therefore, for these spaces, this condition is  equivalent to similar conditions which have been considered by other authors with different names (``uniform compactness of the family of $r$-balls'' in \cite{Gr81}; ``geometrical boundedness'' in \cite{DY}, etc.).

\vspace{2mm}

The following remarkable version of the Margulis' Lemma, due to Breuillard-Green-Tao,  is another important consequence of a packing condition at some fixed scale. It clarifies the structure of the groups  $\overline{\Gamma}_r(x)$ for small $r$, which are sometimes called the {\em ``almost stabilizers''}. We decline it for geodesically complete CAT$(0)$-spaces.

\begin{prop}[\textup{\cite[Corollary 11.17]{BGT11}}]
	\label{prop-Margulis-nilpotent}${}$\\
	Given $P_0,r_0 >  0$, there exists $\varepsilon_0 = \varepsilon_0(P_0,r_0) > 0$ such that the following holds.
	Let $X$ be a proper, geodesically complete, $(P_0,r_0)$-packed, $\textup{CAT}(0)$-space and let $\Gamma$ be a discrete subgroup of $\textup{Isom}(X)$: then, for every $x\in X$ and  every $0\leq \varepsilon \leq \varepsilon_0$, the almost stabilizer $\overline{\Gamma}_{\varepsilon}(x)$ is virtually nilpotent.
\end{prop}

\noindent We will often refer to the constant $\varepsilon_0=\varepsilon_0 (P_0,r_0)$ as the {\em Margulis' constant}.
The conclusion of Proposition \ref{prop-Margulis-nilpotent} can be improved for cocompact groups, as in this case the group $\overline{\Gamma}_{\varepsilon}(x)$ is {\em virtually abelian}   (cp. \cite[Theorem II.7.8]{BH09}; indeed, a cocompact group of a CAT$(0)$-space is always semisimple).

%\noindent We mention the fact that the function $V$ appearing in Proposition \ref{prop-packing} can be chosen so that $\lim_{R \to 0}V(R) = 0$, if the space is not reduced to a point:
%\begin{lemma}
%	\label{lemma-small-volumes}
%	For all $\eta > 0$ there exists $R := R(P_0,r_0,\eta) > 0$ depending only on $P_0, r_0$ and $\eta$ such that 
%	$$\mu_X(\overline{B}(x,r)) \leq \eta$$
%	%the natural measure of 
%	for any ball of radius $r\leq R$ of any complete, geodesically complete, \textup{CAT}$(0)$-space $X$ which is $(P_0,r_0)$-packed, provided $X$ is not a point.
%\end{lemma}
%\begin{proof}
%	If the thesis is not true then we can find non-trivial, complete, geodesically complete $(P_0,r_0)$-packed, CAT$(0)$-spaces $X_j$ such that $\mu_{X_j}(\overline{B}(x_j, \frac{1}{j})) > \eta$ for some $x_j\in X_j$. By \cite{CavS20}, Theorem 6.1 we can suppose that the $X_j$'s converge in the pointed Gromov-Hausdorff sense to a non-trivial, complete, geodesically complete, CAT$(0)$-space $X_\infty$ which is still $(P_0,r_0)$-packed. Moreover the balls $\overline{B}(x_j, \frac{1}{j})$ converge to a point $x_\infty \in X_\infty$. By \cite{LN19}, Theorem 12.1 we obtain the following contradiction:
%	$$0 = \mu_{X_\infty}(x_\infty) = \lim_{\varepsilon\to 0}\limsup_{j\to +\infty} \mu_{X_j}\left(\overline{B}\left(x_j, \frac{1}{j} + \varepsilon\right)\right) > \eta.$$
%\end{proof}

%\subsection{{\color{red}Margulis' Lemma and thin actions}}
% {\color{red} define $\Sigma(x,r)$  and   $\Gamma(x,r)$: OVUNQUE $>>$  $\Sigma_r(x)$  and   $\Gamma_r(x)$ \\

\subsection{Crystallographic groups in the Euclidean space}	${}$
	
\noindent  We will denote points in $\mathbb{R}^k$ by a bold letter {\bf v},  and the origin by $\bf O$. \\
Among CAT$(0)$-spaces, the Euclidean space ${\mathbb R}^k$ and its discrete groups play a special role.
A \emph{crystallographic group} is a discrete, cocompact group $G$ of isometries of some $\mathbb{R}^k$. The simplest and most important of them, in view of Bieberbach's Theorem, are {\em Euclidean lattices}: i.e. 
%Among crystallographic groups, a special role is played by lattices.A \emph{lattice} of $\mathbb{R}^k$ is a 
free abelian crystallographic groups.
It is well known that a lattice must act by translations on $\mathbb{R}^k$ (see for instance \cite{farkas}); so, alternatively, a lattice $\mathcal{L}$ can be seen as the set of linear combinations with integer coefficients of $k$ independent vectors $\bf b_1, \ldots, \bf b_k$ (we will make no difference between a lattice and this representation). 
 The integer $k$ is also called the {\em rank} of the lattice.\\
A {\em basis} $\mathcal{B} = \lbrace \bf b_1,\ldots,\bf b_k\rbrace$ of a lattice $\mathcal{L}$ is a set of $k$ independent vectors that generate $\mathcal{L}$ as a group. There are many geometric invariants classically associated to a lattice $\mathcal{L}$, we will need just two of them:\\
-- the \emph{covering radius}, which is defined as
 \vspace{-4mm}

	$$\rho(\mathcal{L}) = \inf\left\lbrace r > 0 \text{ s.t. } \bigcup_{\bf v \in \mathcal{L}} \overline{B}({\bf v},r) = \mathbb{R}^k\right\rbrace$$
%	Given a lattice  we denote by $\lambda(\mathcal{L})$ 
-- the \emph{shortest generating radius}, that is
	$$\lambda(\mathcal{L}) = \inf\lbrace r > 0 \text{ s.t. } \mathcal{L} \text{ contains } k \text{ independent vectors of length} \leq r\rbrace.$$
	
%Using \eqref{eq-def-codiameter} and the triangle inequality it 
\noindent Notice that, by the triangle inequality, any lattice  $\mathcal{L}$ is $2\rho(\mathcal{L})$-cocompact.
By definition it is always possible to find a basis $\mathcal{B}= \lbrace \bf b_1,\ldots,\bf b_k\rbrace$ of $\mathcal{L}$ such that $\Vert \bf b_1 \Vert \leq \cdots \leq \Vert b_k \Vert = \lambda(\mathcal{L})$;
%for every $i$; 
this is called a \emph{shortest basis} of $\mathcal{L}$.
The shortest generating radius and the covering radius are related as follows:
	\begin{equation}
		\label{eq-lattice-relation}
		\rho(\mathcal{L}) \leq \frac{\sqrt{k}}{2} \cdot \lambda(\mathcal{L}).
	\end{equation}
	For our purposes, the content of the famous Bieberbach's Theorems can be stated as follows.
	\begin{prop}[Bieberbach's Theorem]
		\label{prop-Bieberbach}${}$

\noindent There exists $J(k)$, only depending on $k$, such that the following holds true.
	For every crystallographic group $G$ of $\mathbb{R}^k$ 
	%$G < \textup{Isom}(\mathbb{R}^k)$, 
	the  subgroup ${\mathcal L}(G) = G \cap \textup{Transl}({\mathbb R}^k)$ is a normal subgroup of index at most $J(k)$, in particular a lattice. 
	\end{prop}
\noindent Here $\textup{Transl}({\mathbb R}^k)$ denotes the normal subgroup of translations of ${\mathcal E}_k=\textup{Isom}(\mathbb{R}^k)$.
	The subgroup $\mathcal{L}(G)$ is called the \emph{maximal lattice} of $G$.
It is well-known that every lattice ${\mathcal L}< G$ of rank $k$ has finite index in $G$.

\subsection{Virtually abelian groups}	
${}$

\noindent Recall that the (abelian, or Pr\"ufer)  {\em rank} 
of an abelian group $A$, denoted $\text{rk} (A)$,   is the maximal cardinality of a subset $S \subset A$
% $\{a_1, ..., a_k\}$
of $\mathbb{Z}$-linear independent elements. We extend this definition to   {\em virtually abelian groups} $G$, defining rk$(G)$ as the rank of every free abelian subgroup $A$ of finite index in $G$:
notice that if $A'$ is a finite index subgroup of an abelian group $A$, then $A$ and $A'$   have same rank, so rk$(G)$ is well defined. One can equivalenty define rk$(G)$ as the rank of every {\em normal}, free abelian subgroup $A$ of finite index in $G$, since every finite index subgroup of $G$ contains a normal, finite index subgroup.
% (as if $A_1, A_2$ are finite index abelian subgroups of $G$, then $A_1 \cap A_2$ is a finite index, abelian subgroup of both $A_1$ and $A_2$, hence $\text{rk} (A_1)=\text{rk} (A_1\cap A_2) =\text{rk}(A_2)$).
It is easy to show that   the abelian rank is monotone on subgroups. 
\vspace{1mm}

\noindent If $A$ is a discrete, finitely\ generated, semisimple free abelian group of isometries of a CAT$(0)$-space $X$, then its {\em minimal set} 
 \vspace{-3mm}
 
$$\text{Min}(A) := \bigcap_{a\in A}\text{Min}(a)$$
 is not empty and splits isometrically as $Z\times \mathbb{R}^k$ where $k=\textup{rk}(A)$. This is the main content of the {\em Flat Torus Theorem} (see \cite[Theorem II.7.1]{BH09}). \\
 We recall  some additional facts about the identification $\text{Min}(A) = Z\times \mathbb{R}^k$, which we will freely use later:
 \vspace{1mm}
 
(a) the abelian group $A$ acts as the identity on the factor $Z$,  and cocompactly by translations on the Euclidean factor $\mathbb{R}^k$;
 \vspace{1mm}
 
 (b) writing $x=(z,v) \in \text{Min}(A)$, the slice $\{ z \} \times \mathbb{R}^k$ coincides with the convex closure $\textup{Conv}(Ax)$ of the orbit $Ax$; 
 \vspace{1mm}
  
(c) one has $\textup{Conv}(A'x) = \textup{Conv}(Ax)$ for every finite index subgroup $A' < A$ and every $x\in \textup{Min}(A)$. 
 \vspace{1mm}
%\begin{lemma}
%	\label{lemma-flat-torus}	Let $X$ be a complete \textup{CAT}$(0)$-space and let $A$ be a discrete, semisimple, free abelian group of isometries of $X$. Then:
%	\begin{itemize}
%		\item[(i)] for all $x\in \textup{Min}(A)$ the convex closure $\textup{Conv}(Ax)$ is isometric to $\mathbb{R}^k$ and $A$ acts as a lattice on $\textup{Conv}(Ax)$;
%		\item[(ii)] for all $x,x'\in \textup{Min}(A)$ the sets $\textup{Conv}(Ax)$ and $\textup{Conv}(Ax')$ are parallel\\
%{\color{red}		(that is,	$\exists y,y' \in Y$ such that 	$\textup{Conv}(Ax)$ and $\textup{Conv}(Ax')$  identify respectively with   $\{ y \} \times \mathbb{R}^k$ and $\{ y' \} \times \mathbb{R}^k$, under the above splitting);}
%		\item[(iii)] $\textup{Conv}(A'x) = \textup{Conv}(Ax)$ for every finite index subgroup $A' < A$ and every $x\in \textup{Min}(A)$.
%	\end{itemize}
%\end{lemma}
%\begin{proof}
%	Points (i) and (ii) follow from \cite[Theorem II.7.1]{BH09}, part (3) and part (2) respectively. Let us prove (iii).

\noindent The last assertion follows from the fact that 
%$\text{Min}(A) \subset \text{Min}(A')$,
%since $A' \subseteq A$ then $x\in \text{Min}(A')$ too, 
$\textup{Conv}(A'x)\subseteq \textup{Conv}(Ax)$ and  are both isometric to $\mathbb{R}^k$, so they necessarily coincide.
%\end{proof}

\vspace{1mm}
As a direct consequence of the Flat Torus Theorem we have the following property for virtually abelian isometry groups of CAT$(0)$-spaces, that we will often use.

\begin{lemma}
	\label{lemma-finite index}
	Let $X$ be a proper $\textup{CAT}(0)$-space and let $G_0 <G$	 be discrete, semisimple, virtually abelian groups of isometries of $X$. Then $[G:G_0]$ is finite if and only if $G$ and $G_0$ have same rank.
\end{lemma}

\begin{proof} The implication $[G:G_0] < \infty \Rightarrow \text{rk} (G)=\text{rk} (G_0)$ is trivial, as every free abelian finite index subgroup $A<G_0$ is also a finite index subgroup of $G$. \linebreak
To show the converse implication, assume that $\text{rk}(G_0)=\text{rk}(G)=k$, and let $A$ be a rank $k$,  free abelian, finite index normal subgroup of $G$. 
Consider the (free abelian) subgroup $A_0=A \cap G_0$ of $G_0$, and notice that we have $\text{rk}(A_0)=\text{rk}(G_0)=k$, since  also 
$[G_0 :  A_0]= [ G : A ] <\infty$.
Now, both $A_0$ and $A$ act faithfully on the Euclidean factor of $\text{Min}(A)=Z \times {\mathbb R}^k$ (they do not contain elliptics since they are free, and act as the identity on $Z$). Therefore their projections $p(A_0) < p(A)$ on $\text{Isom}({\mathbb R}^k)$ are both rank $k$ Euclidean lattices, hence $ [A:A_0]=[p(A) : p(A_0)] < \infty$. But then we deduce that $[G:G_0] \leq [G:A_0]= [G:A][A:A_0]<\infty$.
\end{proof}

The following generalization of the Flat Torus Theorem  is classical.  

\begin{prop}[\textup{\cite[Corollary II.7.2]{BH09}}]
	\label{prop-virtually-abelian}${}$
	
\noindent	Let $X$ be a proper $\textup{CAT}(0)$-space, and let $G$ be a discrete, semisimple, virtually abelian group of isometries of $X$ of rank $k$. 
%Assume that  $\textup{rk}(G)=k$: 
Then,   there exists a closed, convex, $G$-invariant subset $C(G)$ of $X$ which splits as $Z\times \mathbb{R}^k$, satisfying the following properties:
	\begin{itemize}
		\item[(i)] every  $g\in G$ preserves the product decomposition and acts as the identity on the first component;
		\item[(ii)] the image $G_{\mathbb{R}^k}$ of $G$ under the projection  $G \rightarrow \textup{Isom}(\mathbb{R}^k)$ is a crystallographic group. 
	\end{itemize}
\end{prop}

\noindent Given  a generating set $S$ of $G$, the following statement explains  how to construct  a generating set for the maximal lattice ${\mathcal L}(G_{\mathbb{R}^k}) $ with words on $S$ of bounded length. This will be used later in the proof of Theorem \ref{theo-splitting-weak}:

\begin{lemma}
	\label{lemma-virtually-abelian}
	Same assumptions as in Proposition \ref{prop-virtually-abelian} above.	\\
	If  $S$ is a symmetric, finite generating set for $G$ containing the identity, then there exists a subset $\Sigma \subseteq S^{4J(k)+2}$ whose projection  $\Sigma_{\mathbb{R}^k}$ on $\textup{Isom}(\mathbb{R}^k)$ generates  $\mathcal{L}(G_{\mathbb{R}^k})$, where $J(k)$ is the constant of Proposition \ref{prop-Bieberbach}.
	% $<S'> \leq <S^{4J+2}> \leq G$, quindi preserva $W \times \mathbb{R}^k$} 
\end{lemma}

\noindent Here $S^n \subset G$
%$S^n:=\overbrace{S\cdots S}^{n \text{ times}}$ 
 denotes the subset  of all products of at most $n$ elements of $S$;  by definition, this coincides with $G \cap \overline{B}(e,n)$, where  $\overline{B}(e,n)$ is the closed ball of radius $n$, centered at the identity $e$, in the Cayley graph $\text{Cay}(G,S)$.

\begin{proof}[Proof of Lemma \ref{lemma-virtually-abelian}]
%There exists $J = J(k)$ depending only on $k$  such that   a subgroup of $G$ of index   $\leq J$   projects on the maximal lattice $\mathcal{L}$ of translations of $G''$; this follows directly from Bieberbach's Theorem. 
Call $p \colon G \to G_{\mathbb{R}^k}$  the projection on Isom$(\mathbb{R}^k)$. By Proposition \ref{prop-Bieberbach} the normal subgroup $\mathcal{L}(G_{\mathbb{R}^k})$ has  index at most $J(k)$ in $G_{\mathbb{R}^k}$, so  the normal subgroup $G':=p^{-1}(\mathcal{L}(G_{\mathbb{R}^k}))$ has index at most $J(k)$ in $G$. 
The group $G$ acts discretely by isometries on Cay$(G,S)$ with codiameter $1$. \linebreak 
So, by Lemma \ref{lemma-cocompact-finite-index} the group $G'$ acts on Cay$(G,S)$ with codiameter at most $2J(k)+1$. 
As recalled before Lemma \ref{lemma-cocompact-finite-index}, $G'$ is therefore generated by the subset $\Sigma:=G' \cap S^{4J(k)+2}  $, made of  the elements of $G'$ displacing the point $e$ in the Cayley graph Cay$(G,S)$ at most by $4J(k)+2$.
%Therefore $p^{-1}(\mathcal{L}(G_{\mathbb{R}^k}))$ is generated by a subset $\Sigma \subseteq S^{4J(k)}$ by Lemma \ref{lemma-generating-finite-index}. 
It follows that the projection   $\Sigma_{\mathbb{R}^k}=p(\Sigma)$ generates $\mathcal{L}(G_{\mathbb{R}^k})$.
\end{proof}

Finally, remark that given a discrete, semisimple, virtually abelian group $G$ of isometries of a proper CAT$(0)$-space, there can be several closed, convex, $G$-invariant subsets $C(G)$ satisfying the conclusions of Proposition \ref{prop-virtually-abelian}. What is uniquely associated to $G$ is a subset in the boundary of $X$.

\begin{prop}
	\label{prop-trace-infinity}
	%Let $X$ be a complete,  \textup{CAT}$(0)$-space and $G < \textup{Isom}(X)$ be discrete, semisimple and virtually abelian with $\textup{rk}(G)=k$. 
	Same assumptions as in Proposition \ref{prop-virtually-abelian}. \\ Then there exists a closed, convex, $G$-invariant subset of $\partial X$, denoted $\partial G$, 
%	$Z \subset \partial X$ 
which is isometric to $\mathbb{S}^{k-1}$ and has the following properties:
\begin{itemize}
		\item[(i)] $\partial G = \partial \textup{Conv}(Ax)$, for every free abelian finite index subgroup $A$ of $G$ and every $x\in \textup{Min}(A)$;
		\item[(ii)]  for every subgroup $G' < G$ we have  $\partial G' \subseteq \partial G$; moreover $\partial G' = \partial G$ if  $\textup{rk}(G')=\textup{rk}(G)$.
	\end{itemize}
%	Let $A$ be any free abelian finite index subgroup of $G$ and $x\in \textup{Min}(A)$. Then $Z = \partial \textup{Conv}(Ax)$.
%   Let $G'<G$ be any finite index subgroup and let $C(G') = W' \times \mathbb{R}^k$ be any closed, convex, $G'$-invariant subset of $X$ satisfying Proposition \ref{lemma-virtually-abelian} applied to $G'$. Then for all $w'\in W'$ it holds $\partial (\lbrace w' \rbrace \times \mathbb{R}^k) = Z$.
\end{prop}
\noindent The closed subset $\partial G$ provided by this proposition will be called the \emph{trace at infinity} of the virtually abelian group $G$.
\begin{proof}
	We fix a free abelian, normal subgroup $A_0 \triangleleft G$ with finite index and   $x_0\in \text{Min}(A_0)$.
	By the Flat Torus Theorem, 
% Lemma \ref{lemma-flat-torus}.(i) 
$\text{Conv}(A_0x_0)$ is isometric to $\mathbb{R}^k$. 
	We set $\partial G:=\partial \textup{Conv}(A_0x_0)$. Clearly $\partial G$ is closed, convex, isometric to $\mathbb{S}^{k-1}$. 
%By Lemma \ref{lemma-flat-torus}.(ii) 
Notice that the set $\partial G$ does not depend on the choice of $x_0\in \text{Min}(A_0)$, since for $x\in \text{Min}(A_0)$ the subsets $\textup{Conv}(A_0x)$ and $\textup{Conv}(A_0x_0)$ can be identified, by the Flat Torus Theorem,  to two parallel slices $\{z \} \times \mathbb{R}^k$ and $\{z_0 \} \times \mathbb{R}^k$, which have the same boundary.  Also, the subset  $\partial G$   is $G$-invariant: in fact, $A_0$ is normal in $G$, so $\text{Min}(A_0)$ is $G$-invariant, therefore
% Hence, for a fixed $x_0\in \text{Min}(A_0)$ we have
	$$g \cdot \partial G= \partial \left( g \cdot \text{Conv}(A_0x_0) \right)= \partial \text{Conv}((gA_0g^{-1}) gx_0) = \partial \text{Conv}(A_0 gx_0) = \partial G$$
because $gx_0 \in \text{Min} (A_0)$.
Again by the Flat Torus Theorem (namely, property (c) recalled before),  if $A<G$ is  another free abelian subgroup of finite index and $x \in  \text{Min}(A)$   then 
 %By Lemma \ref{lemma-flat-torus}.(iii) we have
 	$\text{Conv}(Ax) = \text{Conv}((A\cap A_0) x)= \text{Conv}((A\cap A_0) x_0)= \text{Conv}(  A_0 x_0)$, so  we have $\partial \text{Conv}(Ax) = \partial G$ too. This proves (i).   \\
%\end{proof}
%
%The traces at infinity behave well under inclusion.
%
%\begin{lemma}
%	\label{lemma-trace-infinity-inclusion}
%	Let $X$ be a complete,  \textup{CAT}$(0)$-space and $G < \textup{Isom}(X)$ be discrete, semisimple and virtually abelian and let $G' < G$ be a subgroup. Then $\partial G' \subseteq \partial G$. If moreover $\textup{rk}(G')=\textup{rk}(G)$ then $\partial G' = \partial G$.
%\end{lemma}
%\begin{proof}
\noindent To show (ii), 	it is enough to consider free abelian subgroups $A'<G'$ and $A<G$ of finite index. We can even suppose $A'<A$ up to replacing $A'$ by $A'\cap A$. If $x\in \text{Min}(A)$ then $x\in \text{Min}(A')$, and by the first part of the statement we have 
 $\partial G' = \partial\text{Conv}(A'x) \subseteq \partial\text{Conv}(Ax) = \partial G.$ 
	If moreover $G$ and $G'$ have the same rank then $A'$ is a finite index subgroup of $G$ by Lemma \ref{lemma-finite index}, and $\partial G = \partial\text{Conv}(A'x) = \partial G'$.
\end{proof}

% 
% 
%\begin{proof}
%	Point (i) is exactly {\color{blue} Corollary 7.2, Sec. II of \cite{BH09},} while (ii) follows directly from Bieberbach's Theorem. Let $p\colon G \to G''$ be the projection on the Euclidean factor. The group $p^{-1}(\mathcal{L})$ has index at most $J$ in $G$. Moreover $G$ must be infinite if $k>0$, {\color{blue}because it projects to $\mathcal{L}$, which}  is infinite.
%%  in generale $p^{-1} (  {\cal L}) / ker(p) \cong {\cal L} \cap p(G)$, non $\cong {\cal L}$, ma qui $p(G) \supset  {\cal L}$ 
%	Therefore $p^{-1}(\mathcal{L})$ is {\color{red}a normal subgroup of $G$ and is} generated by a subset $S' \subset S^{4J}$ by Lemma \ref{lemma-generating-finite-index}. Then, clearly    $p(S')$ generates $\mathcal{L}$.
%\end{proof}

%We end this section by recalling a standard property of abelian groups acting on CAT$(0)$-spaces.

\vspace{2mm}
\section{Systole and diastole}
\label{sec-bounds}

In this section we compare some different invariants  of an action of group $\Gamma$   on a CAT$(0)$-space $X$, which are related to the problem of collapsing: the systole and the diastole of the action (and their corresponding free analogues), which play the role of the injectivity radius.\\
Recall that the {\em systole} and the {\em free-systole} of $\Gamma$ {\em at a point} $x\in X$ are defined respectively as
$$\text{sys}(\Gamma,x) := \inf_{g\in \Gamma^*} d(x,gx) , 
\qquad \text{sys}^\diamond(\Gamma,x) := \inf_{g\in \Gamma^\ast \setminus \Gamma^\diamond} d(x,gx),$$
where $\Gamma^* = \Gamma \setminus \lbrace \text{id} \rbrace$ and $\Gamma^\diamond $ is the subset of all elliptic isometries of $\Gamma$.\\
%	$= \lbrace g \in \Gamma \text{ s.t. } g \text{ is elliptic}\rbrace$. 
The  {\em (global) systole} and the  {\em free-systole of} $\Gamma$ are accordingly defined as 
$$\text{sys}(\Gamma,X) = \inf_{x\in X}\text{sys}(\Gamma,x), \qquad 
\text{sys}^\diamond(\Gamma,X) = \inf_{x\in X}\text{sys}^\diamond(\Gamma,x).$$
%{\color{blue} ma la usiamo mai la free systole?}\\
Similarly,  the {\em diastole}   and the  {\em free-diastole of} $\Gamma$ are  defined as  	
$$\text{dias}(\Gamma,X) = \sup_{x\in X}\text{sys}(\Gamma,x) , \qquad 	
\text{dias}^\diamond(\Gamma,X) = \sup_{x\in X}\text{sys}^\diamond(\Gamma,x).$$	

In \cite{CS22}  the authors showed that  dias$(\Gamma, X)>0$ if and only if there exists a {\em fundamental domain for $\Gamma$}; that is, if and only if there exists a point $x_0$ such that the pointwise stabilizer $\text{Stab}_\Gamma(x_0)$ is trivial. The actions satisfying this property will be called {\em nonsingular}, and {\em singular}  when dias$(\Gamma, X)=0$ (with abuse of language, we will often say that the group $\Gamma$ itself is singular or nonsingular).
For a nonsingular action, the {\em Dirichlet domain at $x_0$} is defined as
$$\mathrm{D}_{x_0} = \lbrace y \in X \text{ s.t. } d(x_0,y) < d(x_0,gy) \text{ for all } g \in \Gamma^*\rbrace$$
and is always a fundamental domain for the action of $\Gamma$, see \cite[Prop. 2.9]{CS22}.
Notice that a fundamental domain exists when, for instance, $\Gamma$ is torsion-free, or when  $X$ is a homology manifold, as follows by the combination of Lemma 2.1 and Proposition 2.9 of \cite{CS22}.

\noindent By definition, we have the trivial inequalities:
$$\text{sys}(\Gamma,X) \leq  \text{sys}^\diamond(\Gamma,X) \leq  \text{dias}^\diamond(\Gamma,X)$$
$$\text{sys}(\Gamma,X) \leq  \text{dias}(\Gamma,X) \leq  \text{dias}^\diamond(\Gamma,X).$$

\noindent The following result shows that the free systole and the free diastole are for small values quantitatively equivalent, provided one knows an a priori bound on the diameter of the quotient. Moreover, for nonsingular actions, both are quantitatively equivalent to the diastole.

\begin{theo}
\label{prop-vol-sys-dias}
Given $P_0,r_0,D_0$,  there exists   $b_0 = b_0(P_0,r_0) > 0$ such that the following holds true. 
Let $X$ be a proper, geodesically complete, $(P_0,r_0)$-packed, $\textup{CAT}(0)$-space and let $\Gamma<\textup{Isom}(X)$ be discrete and $D_0$-cocompact. \linebreak Then:
\begin{itemize}
	\item[(i)] 
	$ \textup{sys}^\diamond(\Gamma, X) \geq 
	\left( 1+P_0 \right)^{-\frac{(2D_0+1)}{\min\lbrace \textup{dias}^\diamond(\Gamma,X), r_0 \rbrace}  }$.
	\item[(ii)] If moreover the action of $\Gamma$ on $X$ is nonsingular then 
	$$\textup{dias}(\Gamma, X) \geq b_0 \cdot  \min\lbrace \textup{sys}^\diamond(\Gamma, X),\varepsilon_0\rbrace.$$
\end{itemize}

\noindent	(Here, $\varepsilon_0$ is the Margulis' constant given by Proposition \ref{prop-Margulis-nilpotent}). 
\end{theo}

\noindent Dropping the assumption of nonsingularity.
%on the triviality of some stabilizer,
 (ii) is no longer true: see \cite[Example 1.4]{CS22}, where dias$(\Gamma,X)=0$ while the free systole is positive.
Also, it is easy to convince oneself that the  usual systole is not equivalent, for small values, to the other three invariants: for instance, for every discrete, cocompact action of a group  $\Gamma$ on a  proper CAT$(0)$-space $X$,  one has $\text{dias}(\Gamma,X)>0$ if there exists some point trivial stabilizer, but  clearly $\text{sys}(\Gamma,X)=0$ if $\Gamma$ has torsion.
Finally, notice that the inequality (ii) holds for a constant depending only on $P_0$ and $r_0$, and not on $D_0$; it is not difficult to show that the same is not true for (i).
\vspace{1mm}

 To show the above equivalences, we need two auxiliary facts. \\The first one is a generalization of Buyalo's and Cao-Cheeger-Rong's theorem about the existence of abelian, local splitting structure,  for groups  acting faithfully and geometrically on packed, CAT$(0)$-spaces which are $\varepsilon$-thin for sufficiently small $\varepsilon$.

\begin{prop}[\textup{\cite[Theorem A \& Remark 3.5]{CS22}}]
\label{prop-decomposition-CS}${}$

\noindent Let $P_0,r_0>0$ and fix $0<\lambda \leq \varepsilon_0$, where $\varepsilon_0$ is given by Proposition \ref{prop-Margulis-nilpotent}. There exists a constant $b_0 = b_0(P_0,r_0) > 0$ such that the following holds true. Let $X$ be a proper, geodesically complete, $(P_0,r_0)$-packed, $\textup{CAT}(0)$-space and let $\Gamma< \textup{Isom}(X)$ be discrete and cocompact. If  \textup{dias}$(\Gamma, X) \leq b_0\cdot \lambda$ then
\vspace{-3mm}
 
$$X = \bigcup_{g\in \Gamma^*, \, \ell(g) \leq \lambda} \textup{Min}(g).$$ 
%\st{If moreover $X$ is a homology manifold then the conclusion can be strengthened by taking the union over all non-elliptic $g \in \Gamma^\ast$ with 
%	$\ell(g) \leq \lambda$.}
%{\color{red} da capire se questa parte barrata ci serve.}\\
%{\color{red}(We can choose $b_0=...$)}	
\end{prop}

\noindent The second fact is the following result, which propagates the smallness of the systole for torsion-free cyclic groups from point to point.

\begin{prop}
\label{lemma-Sylvain}
Let $P_0,r_0,R> 0$ and $0<\varepsilon \leq r_0$. Then there exists $\delta(P_0,r_0, R, \varepsilon)  > 0$ with the following property. Let $X$ be a proper, geodesically complete, $(P_0,r_0)$-packed, \textup{CAT}$(0)$-space. If $g$ is an  isometry of $X$ with infinite order and $x$ is a point of $X$ such that $d(x,gx) \leq  \delta(P_0,r_0, R, \varepsilon)$, then for every $y\in X$ with $d(x,y)\leq R$ there exists $m\in \mathbb{Z}^*$ such that $d(y,g^m y) \leq \varepsilon$.\\
(We can choose, explicitely,  $\delta(P_0,r_0, R, \varepsilon)= (1+P_0)e^{-\frac{(2R+1)}{\varepsilon}}$).
\end{prop}
\begin{proof} 
This follows immediately from \cite[Proposition 4.5.(ii)]{CavS20bis}, where we expressed this property in terms of the distance between generalized Margulis domains\footnote{Notice that in  \cite{CavS20bis} we were in the setting of  Gromov-hyperbolic  spaces  with a geodesically complete convex geodesic bicombing (in particular, Gromov-hyperbolic, geodesically complete CAT$(0)$-spaces); but in that proof we never used the hyperbolicity.}.
\end{proof}

\begin{proof}[Proof of Theorem 	\ref{prop-vol-sys-dias}]
%{\color{olive} Se non parliamo di volume questa parte va tolta.}
%{\color{red}Let us prove (i). Let $\varepsilon= \text{dias}(\Gamma,X) > 0$ and assume that $x_0\in X$ is a point where the diastole is realized, that is $d(x_0,gx_0) \geq \varepsilon$ for all $g\in \Gamma^*$. Let $\mathcal{D}$ be any fundamental domain. Fix any $0<\eta < \min\lbrace\frac{\varepsilon}{4}, 1\rbrace$. By local finiteness of the covering $\lbrace g\overline{\mathcal{D}}\rbrace_{g\in \Gamma}$ we can find a finite number of distinct elements $g_1,\ldots,g_n \in \Gamma$ such that $$\overline{B}\left(x_0,\eta \right) \subseteq \bigcup_{i=1}^n \overline{B}\left(x_0,\eta \right)\cap g_i(\overline{\mathcal{D}}).$$
%The sets $E_i = g_i^{-1}( \overline{B}\left(x_0,\eta \right)\cap g_i(\overline{\mathcal{D}})) \subseteq \overline{\mathcal{D}}$ are all disjoint. Indeed if $y\in E_i\cap E_j$ then $g_iy,g_jy\in \overline{B}\left(x_0,\eta \right)$, so $d(y,g_j^{-1}g_iy)\leq 2\eta$ and $d(x_0, g_j^{-1}g_i x_0) \leq 4\eta < \varepsilon$, a contradiction. So
%$$\mu_X(\overline{\mathcal{D}}) \geq \sum_{i=1}^n \mu_X(E_i) = \sum_{i=1}^n \mu_X(g_iE_i) \geq \mu_X\left(\overline{B}(x,\eta) \right) \geq a_0\cdot {\eta}^{\frac{P_0}{2}}$$
%}
%for the constant  $a_0$ given by Proposition \ref{prop-packing}.(iii), depending  on $P_0,r_0$. {\color{red} The thesis follows by arbitrariness of $\eta$}.\\
Assume that 
$\min\lbrace \text{dias}^\diamond (\Gamma ,  X), r_0 \rbrace > \varepsilon$. 
By definition there exists $x_0 \in X$ such that for every hyperbolic isometry $g \in \Gamma$ one has  $d(x_0,gx_0) >\varepsilon$. 
Now, if sys$^\diamond (\Gamma, X) \leq (1+P_0)e^{-\frac{(2D_0+1)}{\varepsilon}} =: \delta$, we could find $x \in X$ and a hyperbolic $g \in \Gamma$ such that $d(x,gx) \leq \delta$. 
By Proposition \ref{lemma-Sylvain}, for every $y \in B(x,D_0)$ there would exists a non trivial power $g^m$ satisfying $d(y,g^my) \leq \varepsilon$. 
But then, since the action is $D_0$-cocompact we could find a conjugate $\gamma$ of $g^m$ (thus,  a hyperbolic isometry) such that $ d(x_0,\gamma x_0)\leq \varepsilon$, a contradiction. The conclusion follows by the arbitrariness of $\varepsilon$.\\
To see (ii), take any $\lambda <\min\lbrace \text{sys}^\diamond(\Gamma, X), \varepsilon_0\rbrace$. If $\text{dias}(\Gamma,X) \leq b_0\cdot \lambda  $, where $b_0$ is the constant given by Proposition \ref{prop-decomposition-CS}, we conclude that 
$X = \bigcup_{g}\text{Min}(g)$\linebreak
where $g$ runs over all nontrivial elements with  $\ell(g) \leq \lambda$.
%	$$X = \bigcup_{g\in \Gamma^\ast,\, \ell(g) \leq \lambda}\text{Min}(g).$$
	But by definition $\ell(g) > \lambda$ for all hyperbolic $g$, so $X = \bigcup_{g\in \Gamma^\diamond}\text{Min}(g).$
	Hence, $\text{dias}(\Gamma,X) = 0$, contradicting the nonsingularity of $\Gamma$. As   $\lambda$ is arbitrary, this proves (ii).
\end{proof}

%As a consequence of Theorem \ref{prop-vol-sys-dias}, Theorem \ref{cor-intro-finiteness-of-groups}    stated in the introduction also holds under the alternative assumption that the diastole or the free diastole  are  sufficiently small.

%\vspace{2mm}
\section{The splitting theorem}
\label{sec-splitting}
In this section we will prove Theorem \ref{theo-intro-splitting}, actually a stronger, parametric version given by Theorem \ref{theo-splitting-weak} below.
To set the notation, recall that for a proper, geodesically complete,  CAT$(0)$-space  $X$ which is $(P_0, r_0)$-packed we have an upper bound on the dimension of $X$ given by Proposition \ref{prop-packing}
$$\dim (X)\leq n_0 = P_0/2$$ 
and a Margulis's constant $ \varepsilon_0$ (only depending  on $P_0, r_0$)    given by Proposition \ref{prop-Margulis-nilpotent}.
Also, recall the constant $J(k)$ given by Proposition \ref{prop-Bieberbach}, and  define   \vspace{-3mm}

$$J_0 := \max_{k\in \lbrace 0,\ldots,n_0\rbrace}J(k) +1$$
which also clearly  depends only on $n_0$, so ultimately only on $P_0$.  \\
Finally, for a discrete subgroup  $\Gamma < \textup{Isom}(X)$ 
% and $D_0$-cocompact.
recall the definition (\ref{defgamma}) of the subgroup  $\overline{\Gamma}_{r} (x) < \Gamma$ generated by  $\overline{\Sigma}_r (x)$ given in Section \ref{subsection-isometries}.

\begin{theo}
	\label{theo-splitting-weak}
	Given positive constants  $P_0,r_0,D_0$, 
	%let $\varepsilon_0=\varepsilon_0(P_0,r_0)$ be the Margulis' constant.
	there exists a   function 
	$\sigma_{P_0,r_0,D_0}: (0,\varepsilon_0] \rightarrow  (0,\varepsilon_0]$   (depending only on the parameters $P_0,r_0,D_0$) such that the following holds. 
%	$0<\lambda \leq \varepsilon_0$. Then there exists $s=s(\lambda) = s(P_0,r_0,D_0,\lambda) > 0$ such that the following holds true. 
Let $X$ be a proper, geodesically complete, $(P_0,r_0)$-packed, $\textup{CAT}(0)$-space,   and $\Gamma < \textup{Isom}(X)$ be discrete and $D_0$-cocompact. 
For every chosen $\varepsilon \in (0, \varepsilon_0]$, if $\textup{sys}^\diamond(\Gamma,X) \leq \sigma_{P_0,r_0,D_0}(\varepsilon)$ then:
	\begin{itemize}[leftmargin=9mm] 
\item[(i)] the space $X$ splits isometrically  as $Y \times \mathbb{R}^k$, with $k\geq 1$, and this splitting is $\Gamma$-invariant;
		\item[(ii)] there exists ${\varepsilon^\ast} \in (\sigma_{P_0,r_0,D_0} (\varepsilon), \varepsilon)$ 
%		$s\leq \eta < \lambda$ 
such that the rank of the virtually abelian subgroups $\overline{\Gamma}_{{\varepsilon^\ast}}(x)$ is  exactly $k$, for all $x\in X$; 
\item[(iii)]   the traces at infinity $\partial \overline{\Gamma}_{{\varepsilon^\ast}}(x)$ equal  the boundary ${\mathbb S}^{k-1}$ of  the convex subsets  $\lbrace y \rbrace \times \mathbb{R}^k$, for all $x \in X$ and all $y\in Y$;
		\item[(iv)] for every $x\in X$ there exists $y\in Y$ such that $\overline{\Gamma}_{\varepsilon^\ast}(x)$ preserves $\lbrace y \rbrace \times \mathbb{R}^k$;   
% Denote by $Z\subseteq \partial X$ the boundary of all sets $\lbrace y \rbrace \times \mathbb{R}^k$, $y\in Y$.
		 \item[(v)]  the projection 
%$p_{\mathbb{R}^k}$ 
of $\overline{\Gamma}_{\varepsilon^\ast}(x)$ on $\textup{Isom}(\mathbb{R}^k)$ is a crystallographic group, whose maximal lattice is generated by the projection of a subset $\Sigma \subset \Sigma_{4J_0\cdot{\varepsilon^\ast}}(x)$; 
%generates a subgroup $L_\eta(x) < \overline{\Gamma}_\eta(x)$ whose projection on $\textup{Isom}(\mathbb{R}^k)$ is the maximal lattice of $p_{\mathbb{R}^k}(\overline{\Gamma}_\eta(x))$;
		 \item[(vi)]  the closure of the projection of $\overline{\Gamma}_{\varepsilon^\ast}(x)$ on $\textup{Isom}(Y)$ is compact and totally disconnected.
	\end{itemize}		
\end{theo}

\noindent Here, by {\em $\Gamma$-invariant splitting} we mean that every isometry of $\Gamma$ preserves the product decomposition. By of \cite[Proposition I.5.3.(4)]{BH09} we can see $\Gamma$ as a subgroup of $\textup{Isom}(Y) \times \textup{Isom}(\mathbb{R}^k)$. In particular it is meaningful to talk about the projection of $A$ on $\text{Isom}(Y)$ and $\text{Isom}(\mathbb{R}^k)$.

\noindent Observe that Theorem \ref{theo-intro-splitting} is a special case of Theorem \ref{theo-splitting-weak}.(i), for $\varepsilon=\varepsilon_0$, which yields the constant $\sigma_0= \sigma_{P_0,r_0,D_0}(\varepsilon_0)$. We call the integer  $1 \leq k\leq n_0$ 
%\in \lbrace 1,\ldots,n_0\rbrace$ 
given by (ii) the {\emph{${\varepsilon^\ast}$-splitting rank}} of $X$.
\vspace{1mm}

 To prove Theorem \ref{theo-splitting-weak} we need a little of preparation. \\ The first step will be to exhibit a free abelian subgroup of rank $k \geq 1$ which is commensurated in $\Gamma$.
Recall that two subgroups $G_1, G_2 < \Gamma$ are called {\em commensurable} in $\Gamma$ if the intersection $G_1 \cap  G_2$ has finite index in both $G_1$ and $G_2$. A subgroup   $G < \Gamma$ is said to be commensurated in $\Gamma$ if   $G$ and $\gamma G \gamma^{-1}$ are commensurable in $\Gamma$ for every $\gamma \in \Gamma$.
%The {\em commensurator} of a subgroup $G$ in $\Gamma$ is the largest subgroup $Comm_\Gamma (G)$ is the
\vspace{1mm}

%\noindent We   fix some notation for the sequel: $X$ is  a complete, geodesically complete, $(P_0, r_0)$-packed, CAT$(0)$-space  and $\Gamma < \textup{Isom}(X)$ is discrete and $D_0$-cocompact.
%Let  $n_0 = P_0/2$ be the upper bound on the dimension of $X$ given by Proposition \ref{prop-packing}, and let $\varepsilon_0$ be the Margulis constant  given by Proposition \ref{prop-Margulis-nilpotent} (only depending  on $P_0, r_0$). 
%Finally, consider the constant $J(k)$ given by Proposition \ref{prop-Bieberbach}, and  define   \vspace{-3mm}
%
%$$J_0 := \max_{k\in \lbrace 0,\ldots,n_0\rbrace}J(k)$$
%which clearly  depends only on $n_0$, so ultimately only on $P_0$.  
%\vspace{1mm}
%

\noindent 
%the group  $\Gamma_{\varepsilon_0} (x)$ is virtually abelian by Proposition \ref{prop-Margulis-nilpotent}. 
Now, we fix $0<\varepsilon \leq \varepsilon_0$ as in the assumptions of Theorem \ref{theo-splitting-weak}, and we define  inductively the sequence of subgroups  $\overline{\Gamma}_{\varepsilon_i} (x)$ associated to positive numbers 
$$\varepsilon_1 := \varepsilon  > \varepsilon_2 > \ldots > \varepsilon_{2n_0 + 1} > 0$$ as follows:\\
%-- we define   $\varepsilon_0$ to be the Margulis constant of $X$ given by Proposition \ref{prop-Margulis-nilpotent} (which depends only on $P_0$ and $r_0$); \\
-- first, we  apply Proposition \ref{lemma-Sylvain} to $\varepsilon=\varepsilon_1$ and $R=2D_0$ to obtain a smaller $\delta_2 := \delta(P_0,r_0,2D_0,\varepsilon_1)$, and  we set $\varepsilon_2 := \delta_2/4J_0$;\\
% (this constant only depends on $P_0,r_0, D_0$ and $\lambda_1$);\\
-- then, we define inductively $\delta_{i+1} := \delta(P_0,r_0,2D_0,\varepsilon_i)$ by repeatedly applying Proposition \ref{lemma-Sylvain} to $\varepsilon_i$ and $R=2D_0$, and we set $\varepsilon_{i+1} = \delta_{i+1}/4J_0$. \\
Notice that, by construction, each $\varepsilon_i$ depends only on $P_0, r_0, D_0$ and $\varepsilon$.
 \\By Proposition \ref{prop-Margulis-nilpotent}, the subgroups  $\overline{\Gamma}_{\varepsilon_i} (x)$  form a decreasing sequence of virtually abelian, semisimple subgroups for every $x\in X$.

\vspace{1mm}
\noindent We set $\sigma_{P_0,r_0,D_0} (\varepsilon):= \varepsilon_{2n_0 + 1}$ and we will  show that this is the function of $\varepsilon$
for which Theorem \ref{theo-splitting-weak} holds;  it clearly depends only on $P_0,r_0,D_0$ and $\varepsilon$. In what follows, we will write  for short  $\sigma  := \sigma_{P_0,r_0,D_0} (\varepsilon)$.
%For technical reasons we need to deal only with hyperbolic isometries, so we define $\Gamma_{\varepsilon_i}^\diamond(x)$ the group generated by the hyperbolic isometries of $\Sigma_{\varepsilon_i}(x)$.

\begin{lemma}
\label{lemma-constant-rank}
If 
%%$\textup{sys}^\diamond(\Gamma,X)  
%${\color{red}\text{sys}^\diamond(\Gamma,x)} \leq {\color{red}\sigma }$ \
$\textup{rk} (\overline{\Gamma}_{\sigma}( x)) \geq 1 $
then there exists
% $x_0\in X$ and 
$i\in \lbrace 2,\ldots, 2n_0 \rbrace$ such that
$\textup{rk}(\overline{\Gamma}_{\varepsilon_{i+1}}( x)) = \textup{rk}(\overline{\Gamma}_{\varepsilon_i}(x )) = \textup{rk}(\overline{\Gamma}_{\varepsilon_{i-1}}(x )) \geq 1$.
\end{lemma}

\begin{proof}
%%we can choose $x_0 \in X$ such that 
%$\overline{\Sigma}_{\color{red} \sigma }(x)$ contains  a hyperbolic isometry. }
The subgroups  $\overline{\Gamma}_{\varepsilon_i}(x)$ are   virtually abelian with 
% are all virtually abelian because of Proposition \ref{prop-Margulis-nilpotent} and 
$\textup{rk}(\overline{\Gamma}_{\varepsilon_i}(x))\geq 1$,   for all $1\leq i \leq 2n_0 +1 $, since they  contain $\overline{\Gamma}_{\sigma}( x)$.
Moreover,  by Proposition \ref{prop-virtually-abelian},    the rank of each $\overline{\Gamma}_{\varepsilon_i}(x)$ cannot exceed the dimension of $X$, which is at most $n_0$. \linebreak
Since the rank decreases as $i$ increases we conclude that  for some  $2\leq i\leq 2n_0$ we have  $\textup{rk}(\overline{\Gamma}_{\varepsilon_{i+1}}(x))  = \textup{rk}(\overline{\Gamma}_{\varepsilon_i}(x)) = \textup{rk}(\overline{\Gamma}_{\varepsilon_{i-1}}(x)) \geq 1$.
\end{proof}

\begin{prop}
\label{prop-commensurated}
If 
%%If $\textup{sys}^\diamond(\Gamma,X)\leq \sigma $ 
%$\textup{sys}^\diamond(\Gamma,x)\leq \sigma $ 
 $\textup{rk} (\overline{\Gamma}_{\sigma}(x)) \geq 1 $
%then there exists a free abelian,  subgroup $A$  of rank $k \geq 1$ which is commensurated in $\Gamma$. Namely, we can choose  $A$ as a subgroup of $\overline{\Gamma}_{\lambda_{i+1}}(x_0)$  with finite index in $\overline{\Gamma}_{\lambda_{i-1}}(x_0)$, for some  $x_0 \in X$ and $i \in \lbrace 2,\ldots, 2n_0 \rbrace$ (those given by Lemma \ref{lemma-constant-rank}).
then there exists  a  free abelian subgroup $A  < \overline{\Gamma}_{\varepsilon_i}(x )$ of rank $k \geq 1$, 
which has finite index in  $\overline{\Gamma}_{\varepsilon_{i-1}}(x )$ and is commensurated in $\Gamma$,
%for some $x_0 \in X$ and 
for some $i \in \lbrace 2,\ldots, 2n_0 \rbrace$.
\end{prop}

\begin{proof}
Consider   the virtually abelian groups $\overline{\Gamma}_{\varepsilon_{i+1}}(x ) < \overline{\Gamma}_{\varepsilon_i}(x ) < \overline{\Gamma}_{\varepsilon_{i-1}}(x )$ which have the same rank $k\geq 1$, 
given by Lemma \ref{lemma-constant-rank}, for some 
% $x_0 \in X$ and 
$2 \leq i \leq 2n_0$
%$i \in \lbrace 2,\ldots, 2n_0 \rbrace$ 
and let $A_{\varepsilon_{j}}(x ) < \overline{\Gamma}_{\varepsilon_{j}}(x )$  be   free abelian, finite index subgroups of rank $k$,  for $j=i-1,i,i+1$.
Notice that $A_{\varepsilon_{i+1}}(x )$ has finite index also in  $\overline{\Gamma}_{\varepsilon_{i}}(x )$,
%(and in $\overline{\Gamma}_{\lambda_{i-1}}(x_0)$ as well), 
by Lemma \ref{lemma-finite index}.
Then, let $C_{i+1}=Z_{i+1} \times {\mathbb R}^{k}$ be the $\overline{\Gamma}_{\varepsilon_{i+1}}(x )$-invariant, convex subset  of $X$ given by Proposition  \ref{prop-virtually-abelian}, applied to $\overline{\Gamma}_{\varepsilon_{i+1}}(x )$, 
and  call $\bar \Gamma_{i+1}$  the image of $\overline{\Gamma}_{\varepsilon_{i+1}}(x )$ under the projection  $p_{i+1}:\overline{\Gamma}_{\varepsilon_{i+1}}(x ) \rightarrow \text{Isom}(\mathbb{R}^{k})$ on the second factor of $C_{i+1}$.
%Let $C_{i}=W_{i} \times {\mathbb R}^{k}$ be a $\overline{\Gamma}_{\color{red}\lambda_{i+1}}(x_0)$-invariant, convex subset of $X$ given by Proposition  \ref{prop-virtually-abelian} applied to $\overline{\Gamma}_{\color{red}\lambda_{i+1}}(x_0)$. 
%Let $A_{\color{red}\lambda_{i+1}}(x_0)$,  $A_{\color{red}\lambda_{i}}(x_0)$ be finite index, free abelian subgroups of  $\overline{\Gamma}_{\color{red}\lambda_{i+1}}(x_0)$, $\overline{\Gamma}_{\color{red}\lambda_{i}}(x_0)$ respectively.
Finally, denote by  ${\mathcal L}_{i+1}$ the  maximal Euclidean lattice of the crystallograhic group $\bar \Gamma_{i+1}$.
%Let $\bar \Gamma_{i+1}$ be the image of $\overline{\Gamma}_{\color{red}\lambda_{i+1}}(x_0)$ under the projection  $p_{i+1}:\overline{\Gamma}_{\color{red}\lambda_{i+1}}(x_0) \rightarrow \text{Isom}(\mathbb{R}^{k})$ on the second factor of $C_{i+1}$. 

%Fnally, simply call  $p: \Gamma_{\varepsilon_{i-1}} (x) \rightarrow \text{Isom}(\mathbb{R}^k)$ the projection. \\
%{\color{red}(since they have rank $k$).}
%it has  rank $k$, as $\bar \Gamma_{i-1}$.
%$\Gamma_i'' < \Gamma_{i-1}''$ be the  associated cristallographic groups of $\mathbb{R}^k$, 
%\\ {\color{red}NECESSARI? SONO $= \bar \Gamma_i$ VISTO CHE AGISCONO COME ID SU $W_i$} \\
%and
\noindent By Lemma \ref{lemma-virtually-abelian} we can find a subset $\Sigma$ of 
$\overline{\Sigma}_{\varepsilon_{i+1}}(x )^{4J_0} \subseteq \overline{\Sigma}_{4J_0\cdot \varepsilon_{i+1}}(x )$
%$=\Sigma_{\delta_i}(x)$	
whose projection  $\Sigma_{\mathbb{R}^k}=p_{i+1}(\Sigma)$ on  Isom$(\mathbb{R}^k)$ generates the lattice  ${\mathcal L}_{i+1}$.  
 In particular every non-trivial element of $\Sigma$ is hyperbolic.\\
Moreover,  by the definition of $\varepsilon_{i+1}  = \delta_{i+1} /4 J_0 $,
%$\delta_i = 4 J_0 \cdot \varepsilon_i$ 
%	 {\color{red} non di $\varepsilon_i$, piuttosto di $\delta_i = 4 J_0 \varepsilon_i$} \\
the following holds:
\vspace{-2mm}

\begin{center}
{\em  $ \forall g\in  \Sigma$ 
%	 {\color{red}  quindi $g \in \Sigma_{\delta_i} (x) $   }\\
and $\forall h\in \overline{\Sigma}_{2D_0}(x )$  there exists $m >0$ 
%	\in \mathbb{Z}^*$ 
such that    $h^{-1}g^m h    \in \overline{\Sigma}_{\varepsilon_{i}}(x )$
}
\end{center}

\vspace{-2mm}
%  {\color{red} piuttosto che $hg^mh^{-1}$ \\
\noindent (in fact,  $d(x,gx) \leq \delta_{i+1}=\delta(P_0,r_0,2D_0, {\varepsilon_{i}})$ for every $g \in \Sigma$,  so by Proposition \ref{lemma-Sylvain} there exists $m> 0$ such that \nolinebreak
$ d(x, h^{-1}g^m hx) = d(hx,g^mhx) \leq \varepsilon_{i}$, that is  $h^{-1}g^m h \in \overline{\Sigma}_{\varepsilon_{i}}(x)$). \\
Since  {$\Sigma$ and} $\overline{\Sigma}_{2D_0}(x)$ are  finite sets,  we can then find a positive integer $M$ such that $hg^Mh^{-1} \in {\overline{\Gamma}_{\varepsilon_{i}}}(x)$ for all $g\in \Sigma$ and all $h\in \overline{\Sigma}_{2D_0}(x)$.  \\
Moreover,  we can even choose  $M>0$ so that   $\forall g\in \Sigma$ and $ \forall h\in \overline{\Sigma}_{2D_0}(x)$ 
the elements $g^M$  and  $hg^Mh^{-1}$  belong, respectively,  to the free abelian, finite index subgroups  $A_{\varepsilon_{i+1}} (x)$ and  $A_{\varepsilon_{i}} (x)$ of $\overline{\Gamma}_{\varepsilon_{i}}(x)$. \\
Now, let 
$A < A_{\varepsilon_{i+1}} (x)$ 
% < \overline{\Gamma}_{\lambda_{i+1}}(x_0)}$ 
be the   (free abelian) subgroup  generated by the subset $S=\lbrace g^M  \,\;|\;\, g \in \Sigma \rbrace$. 
We claim that $A$ is commensurated in $\Gamma$.\\
%Now, consider the subset $S=\lbrace g^M  \,\;|\;\, g \in \Sigma \rbrace$ and the   (free abelian) subgroup  $A < A_{\lambda_{i+1}} (x_0)$ 
% < \overline{\Gamma}_{\lambda_{i+1}}(x_0)}$ 
%generated by $S$. 
%$$A= < \lbrace g^M  \text{ s.t. } g \in \Sigma \rbrace  > $$ 
Actually, notice first that $A$ has rank equal to $k$, since its projection $p_{i+1}(A)$ is a subgroup of finite index of the  lattice ${\mathcal L}_{i+1}$ of  ${\mathbb{R}^k}$. Therefore, for all $h\in \overline{\Sigma}_{2D_0}(x)$, the free abelian group $hAh^{-1}$ has also rank $k$, and  is contained in the free abelian group $A_{\varepsilon_{i}}(x)$ of same rank. This implies, again by Lemma \ref{lemma-finite index}, that $A$ and $hAh^{-1}$ have finite index in $A_{\varepsilon_{i}}(x)$, and that  $A\cap  hAh^{-1} $ has finite index in both $A$ and  $hAh^{-1} $. Hence $A$ and  $hAh^{-1}$ are commensurable  for every  $h\in \overline{\Sigma}_{2D_0}(x)$.
%{\color{red} non \`e proprio $S$ che genera un lattice, ma la sua proiezione su ${\mathbb R}^k$: fedele? 
	%senn\`o $rk_{ab}(\langle S \rangle)$ pu\`o essere $>rk_{ab}(\langle S'' \rangle)=k$}\\

% Actually, for every  $h\in \Sigma_{2D_0}(x)$, we have that $p(A)$ and $p(hAh^{-1})$ are sub-lattices of $  {\cal L}_{i-1}$, so $p(A) \cap p(hAh^{-1}) $ is again a lattice of ${\mathbb R}^k$, hence it has finite index in both factors.  This implies that  $A\cap  hAh^{-1} $ has finite index in both $A$ and  $hAh^{-1} $,   hence $A$ and  $hAh^{-1}$ are commensurable  for every  $h\in \Sigma_{2D_0}(x)$.\\
\noindent Now, given $g \in \Gamma$, we can write it as $g=h_1\cdots h_n$ for some $h_i \in \overline{\Sigma}_{2D_0}(x)$ and 
set $g_k := h_1\cdots h_k$. We clearly have  
$$g_k A g_k ^{-1} \cap g_{k+1} A g_{k+1}^{-1}=g_k(A \cap h_{k+1} A  h_{k+1}^{-1})g_k^{-1}$$ 
with $A \cap h_{k+1} A h_{k+1}^{-1}$ of finite index in both factors;  so  $g_k A g_k ^{-1} $ and $g_{k+1} A g_{k+1}^{-1}$ are commensurable. Since commensurability in a group is a transitive relation, this shows that $A$ and $gAg^{-1}$ are commensurable  for all $g \in \Gamma$. 	\\
Finally,  observe that $A<\overline{\Gamma}_{\varepsilon_{i-1}}(x)$ and these groups have same rank, so $A$ has finite index also in $\overline{\Gamma}_{\varepsilon_{i-1}}(x)$, again by Lemma \ref{lemma-finite index}.
\end{proof}

 We need now to recall an additional notion. Given $Z\subseteq \partial X$ we say that a subset $Y\subseteq X$ is $Z$-boundary-minimal if it is closed, convex, $\partial Y = Z$ and $Y$ is minimal with this properties. The union of all the $Z$-boundary-minimal sets is denoted by \textup{Bd-Min}$(Z)$.
A particular case of \cite[Proposition 3.6]{CM09b} reads as follows.

\begin{lemma}
	\label{lemma-split-sphere}
	Let $X$ be a proper $\textup{CAT}(0)$-space and let $Z$ be a closed, convex subset of  $\partial X$ which is isometric to $\mathbb{S}^{k-1}$. Then each $Z$-boundary-minimal subset of $X$ is isometric to $\mathbb{R}^k$ and \textup{Bd-Min}$(Z)$ is a closed, convex subset of $X$ which splits isometrically as $Y\times \mathbb{R}^k$. Moreover $Z$ coincides with the boundary at infinity of all the slices $\lbrace y \rbrace \times \mathbb{R}^k$, for $y \in Y$.
\end{lemma}

A consequence of Lemma \ref{lemma-split-sphere} for a commensurated group $A$ of $\Gamma$ is the following.

\begin{prop}
	\label{prop-commensurated-splitting}
	Same assumptions as in Theorem \ref{theo-splitting-weak}.\\
	If $A$ is a free abelian, commensurated subgroup of $\Gamma$ of rank $k$ then we have $X=\textup{Bd-Min}(\partial A)$ and $X$ splits isometrically and $\Gamma$-invariantly as $Y \times \mathbb{R}^k$. Moreover,  the projection of $A$ on $\textup{Isom}(\mathbb{R}^k)$ is a lattice and the closure of the projection of $A$ on $\textup{Isom}(Y)$ is compact and totally disconnected.\\
	The splitting $X=Y \times \mathbb{R}^k$ satisfies the following properties:
\begin{itemize}[leftmargin=7mm] 
	\item[(i)] the trace at infinity $\partial A$ is $\Gamma$-invariant and coincides with the boundary of each slice $\lbrace y \rbrace \times \mathbb{R}^k$, for all $y \in Y$;
	\item[(ii)] if  $A' < A$ is another free abelian, commensurated subgroup of $\Gamma$ of rank $k'$ then the splittings $X = Y \times \mathbb{R}^k$ and $X= Y' \times \mathbb{R}^{k'}$ associated respectively to $A$ and $A'$ are compatible, i.e. $Y'$ is isometric to $Y \times \mathbb{R}^{k-k'}$.
	\end{itemize}
\end{prop}

%{\color{red} In conclusion we have
%	\begin{theo}
%		Let $P_0,r_0,D_0>0$ and $s_0 = s_0(P_0,r_0,D_0) > 0$ be the constant above. Let $X$ be a complete, geodesically complete, $(P_0,r_0)$-packed, \textup{CAT}$(0)$-space and $\Gamma < \textup{Isom}(X)$ be discrete and $D_0$-cocompact. If $\textup{sys}^\diamond(\Gamma,X) \leq s_0$ then $X$ splits isometrically and $\Gamma$-equivariantly as $Y \times \mathbb{R}^k$, with $k\geq 1$. Moreover every isometry $g\in \Gamma$ with $\ell(g) \leq s_0$ is of the form $(g',g'') \in \textup{Isom}(Y) \times \textup{Isom}(\mathbb{R}^k)$ with $g'$ elliptic.
%	\end{theo}
	\begin{proof}
		The trace at infinity $\partial A$ of $A$ is isometric to $\mathbb{S}^{k-1}$, by Proposition \ref{prop-trace-infinity}. We claim that it is $\Gamma$-invariant. Indeed, if $g\in \Gamma$ then $A\cap gAg^{-1}$ has finite index in both $A$ and $gAg^{-1}$. Then, 
		%Lemma \ref{lemma-trace-infinity-inclusion} and
 the characterization of the trace at infinity given in Proposition \ref{prop-trace-infinity} implies that 
		$$\partial A = \partial \left( A\cap gAg^{-1} \right)= \partial \left(gAg^{-1} \right)= g\partial A.$$
		By Lemma \ref{lemma-split-sphere} applied to $Z=\partial A$ we deduce that $\text{Bd-Min}(\partial A)$ is a closed, convex subset of $X$ which splits isometrically as $Y\times \mathbb{R}^k$, and that  $\partial A$ coincides with the boundary at infinity of all sets $\lbrace y \rbrace \times \mathbb{R}^k$. \\
		Now, each element of $\Gamma$ sends a $\partial A$-boundary-minimal subset into a $\partial A$-boundary-minimal subset because $\partial A$ is $\Gamma$-invariant, therefore $\text{Bd-Min}(\partial A)$ itself is $\Gamma$-invariant. Since $\Gamma$ is  cocompact,  the action of $\Gamma$ on  $X$ is minimal: that is, if $C\neq \emptyset$ is a closed, convex, $\Gamma$-invariant subset of $X$ then $C=X$ (see \cite[Lemma 3.13]{CM09b}). Therefore we deduce that $X=\text{Bd-Min}(\partial A)$, and so $X$ splits isometrically and $\Gamma$-invariantly as $Y\times \mathbb{R}^k$, which proves the first assertion and (i).\\
%Therefore if $x=(y,{\bf v}) \in \text{Min}(A)$ then $\text{Conv}(Ax) = \lbrace y \rbrace \times \mathbb{R}^k$ by Proposition \ref{prop-trace-infinity}. In particular for all $a = (a_1,a_2)\in A$, where $a_1\in \text{Isom}(Y)$ and $a_2\in \text{Isom}(\mathbb{R}^k)$ we know that $a_1$ fixes $y$ and $a_2$ acts by translation on $\mathbb{R}^k$, by Lemma \ref{lemma-flat-torus}. \\
The fact that the projection of $A$ on  $\textup{Isom}(\mathbb{R}^k)$ is a lattice follows from the Flat Torus Theorem. To study the projection of $A$ on $\text{Isom}(Y)$, recall that by Proposition \ref{prop-CM-decomposition} we can split $X$ as $M \times \mathbb{R}^n \times N$, for some $n\geq k$, and  $\textup{Isom}(X)$ as ${\mathcal S}\times {\mathcal E}_n \times {\mathcal D}$, where ${\mathcal S}$ is a semi-simple Lie group with trivial center and without compact factors, ${\mathcal E}_n\cong \text{Isom}(\mathbb{R}^n)$ and ${\mathcal D}$ is totally disconnected. 
Therefore $Y= M \times \mathbb{R}^{n-k} \times N$ and $\textup{Isom}(Y)$ splits as ${\mathcal S} \times  {\mathcal E}_{n-k} \times {\mathcal D}$.
Moreover,   $\text{Min}(A)=Z \times \mathbb{R}^k$ with $Z \subset Y$.
Since $A$ acts as the identity on $Z$,  it follows that the projection of $A$ on  $\textup{Isom}(Y)$ fixes some point $z \in Y$, hence its closure is a compact group.
Finally, Theorem 2.(i) of  \cite{CM19} and the beginning of the proof therein show that   the projection of $A$ on ${\mathcal S}$ is finite and  the projection of $A$ on  ${\mathcal E}_n$ is discrete; as the projection of $A$ on ${\mathcal E}_k$ is a lattice, then also the projection on ${\mathcal E}_{n-k}$ is discrete. This, combined with the fact that ${\mathcal D}$ is totally disconnected,  implies that  the closure of the projection of $A$ on $\text{Isom}(Y)$ is totally disconnected. \\
	Suppose now to have another abelian subgroup $A'<A$ of rank $k'$ which is commensurated in $\Gamma$. 
	Let $X = Y' \times \mathbb{R}^{k'}$ be the splitting associated to $A'$. Let $x\in X$ be a point and write it as $(y,{\bf v}) \in Y \times \mathbb{R}^k$ and $(y',{\bf v'}) \in Y'\times \mathbb{R}^{k'}$. Then, by the first part of the proof and by Proposition \ref{prop-trace-infinity} we have
	$$\partial (\lbrace y' \rbrace \times \mathbb{R}^{k'}) = \partial A' \subseteq \partial A = \partial (\lbrace y \rbrace \times \mathbb{R}^{k}).$$
	 It follows that $\lbrace y' \rbrace \times \mathbb{R}^{k'} \subseteq \lbrace y \rbrace \times \mathbb{R}^{k}$, so the parallel slices associated to $A'$ are contained in the parallel slices associated to $A$. Decomposing $\mathbb{R}^k$ as the orhogonal sum of $\mathbb{R}^{k'}$ and $\mathbb{R}^{k - k'}$, we also deduce that  the sets $\lbrace y \rbrace \times \mathbb{R}^{k - k'}$ are parallel for all $y\in Y$ (since the slices of $A'$ are all parallel). Therefore, $X$ is also isometric to $(Y \times \mathbb{R}^{k - k'}) \times \mathbb{R}^{k'}$, which implies that $Y'$ is isometric to $Y \times \mathbb{R}^{k - k'}$ and proves (ii).
%It is also clearly compact since it fixes $\lbrace y \rbrace$.
\end{proof}

Putting the ingredients all together we can give the

	\begin{proof}[Proof of Theorem \ref{theo-splitting-weak}]  
We  show that the statement holds for 
% $\sigma_{P_0,r_), D_0}(\lambda):=\lambda_{2n_0+1}$ (a function of $\lambda$, only depending on $P_0,r_0,D_0$) and
 ${\varepsilon^\ast}=\varepsilon_i \in (\sigma, \varepsilon)$, where $\varepsilon_i$ is given by Proposition \ref{prop-commensurated}. In fact, since $\textup{sys}^\diamond(\Gamma,X)\leq \sigma $, then there exists $x_0 \in X $ with $\textup{sys}^\diamond(\Gamma,x_0)\leq \sigma $; in particular, $\overline{\Gamma}_{\sigma}(x_0)$   contains a hyperbolic isometry, hence  $\text{rk}(\overline{\Gamma}_{\sigma}(x_0)) \geq 1$. Then, we can   apply Proposition \ref{prop-commensurated}  and find a  free abelian, commensurated subgroup $A_0 <\Gamma$ of  rank $k\geq 1$, with $A_0 \subset \overline{\Gamma}_{\varepsilon_{i+1}}(x_0)$ and  with finite index in $\overline{\Gamma}_{\varepsilon_{i-1}}(x_0)$, 
% for some $x_0 \in X$ and
for  some  $2 \leq i \leq 2n_0$; Proposition \ref{prop-commensurated-splitting} now implies that $X$ splits isometrically and $\Gamma$-invariantly as $Y\times \mathbb{R}^k$, proving (i). 
Moreover, we know that $\partial A_0$ coincides with the boundary at infinity of each slice $\lbrace y \rbrace \times \mathbb{R}^k$. \\
 %with the additional properties that  {\color{purple} the projection of $A$ on $\textup{Isom}(\mathbb{R}^k)$ is a lattice and the closure of the projection of $A$ on $\textup{Isom}(Y)$ is compact and totally disconnected}, 
Let us now study the properties (ii)-(vi) for  the groups  $\overline{\Gamma}_{{\varepsilon_i}}(x)$. 
\\We start proving them for every $x\in X$ such that   $d(x, x_0) \leq D_0$. 
By Proposition \ref{lemma-Sylvain} and by definition of $\varepsilon_i$, for every hyperbolic $g\in \overline{\Sigma}_{\varepsilon_i}(x)$ there  exists $m\in \mathbb{Z}^*$ such that $d(x_0,g^mx_0)\leq \varepsilon_{i - 1}$, 
that is $g^m \in \overline{\Gamma}_{\varepsilon_{i - 1}}(x_0)$. 
Let $A_{\varepsilon_i}(x) < \overline{\Gamma}_{\varepsilon_i}(x)$ be free abelian of finite index, so $\text{rk}(A_{\varepsilon_i}(x)) = \text{rk}\left( \overline{\Gamma}_{\varepsilon_i}(x) \right)$ by definition. Since the set $\overline{\Sigma}_{\varepsilon_i}(x)$ is finite, we can find $M\in \mathbb{Z}^*$ such that $g^M \in A_{\varepsilon_i}(x) \cap \overline{\Gamma}_{\varepsilon_{i-1}}(x_0)$ for every hyperbolic $g\in \overline{\Sigma}_{\varepsilon_i}(x)$. 
So, if we define the subgroup 
$$ A := \langle g^M \,\; | \,\; g\in \overline{\Sigma}_{\varepsilon_i}(x) \text{ hyperbolic}\rangle$$
 we have 	
$A   < A_{\varepsilon_i} (x) \cap \overline{\Gamma}_{\varepsilon_{i-1}}(x_0)$.
Notice that $A$ is again free abelian of rank $k$,  hence it has finite index in $A_{\varepsilon_i}(x)$ and  $\overline{\Gamma}_{\varepsilon_{i}}(x)$, by Lemma \ref{lemma-finite index}.
Thus  $$\text{rk}(\overline{\Gamma}_{\varepsilon_i}(x)) = \text{rk}(A) \leq \text{rk}(\overline{\Gamma}_{\varepsilon_{i-1}}(x_0)) = \text{rk} (A_0) = k.$$
	Moreover by Proposition \ref{prop-trace-infinity}  we have 
	$\partial \overline{\Gamma}_{\varepsilon_i}(x) = \partial A  \subseteq \partial \overline{\Gamma}_{\varepsilon_{i-1}}(x_0) = \partial A_0$.
	Reversing the roles of $x$ and $x_0$ and starting from $\overline{\Gamma}_{\varepsilon_{i+1}}(x_0)$ we obtain the opposite estimate
 $k=\text{rk}(A_0)=\text{rk}(\overline{\Gamma}_{\varepsilon_{i+1}}(x_0)) \leq \text{rk}(\overline{\Gamma}_{\varepsilon_i}(x))$ 
	and $\partial A_0 \subseteq \partial \overline{\Gamma}_{\varepsilon_i}(x)$,	which proves (ii) and (iii) in this case.\\
	By construction $\overline{\Gamma}_{\varepsilon_i}(x) \cap A$ has finite index in both $\overline{\Gamma}_{\varepsilon_i}(x)$ and $A$,  so also the projection $p_{\mathbb{R}^k}(\overline{\Gamma}_{\varepsilon_i}(x))$ on $\textup{Isom}(\mathbb{R}^k)$ is discrete, and the closure of the projection $p_Y(\overline{\Gamma}_{\varepsilon^\ast}(x))$ of $\overline{\Gamma}_{\varepsilon_i}(x)$ on $\text{Isom}(Y)$ is compact and totally disconnected. \linebreak Notice that  $p_{\mathbb{R}^k}(\overline{\Gamma}_{\varepsilon_i}(x))$ is cocompact, so it is a crystallographic group of $ \mathbb{R}^k$; moreover, as $\Sigma_{\varepsilon_i}(x)$ generates $\overline{\Gamma}_{\varepsilon_i}(x)$, then the maximal lattice of $p_{\mathbb{R}^k}(\overline{\Gamma}_{\varepsilon_i}(x))$  is generated by a subset of  $\Sigma_{4J_0\cdot \varepsilon_i}(x)$, by  Lemma \ref{lemma-virtually-abelian}. This proves (v) and (vi). 	
Moreover, 	the  precompact  group $p_Y(\overline{\Gamma}_{\varepsilon_i}(x))$ has a fixed point $y\in Y$ (\cite[Corollary II.2.8.(1)]{BH09}), so $\lbrace y \rbrace \times \mathbb{R}^k$ is preserved by $\overline{\Gamma}_{\varepsilon_i}(x)$, proving (iv). \linebreak
%Moreover  $p_{\mathbb{R}^k}(\overline{\Gamma}_\eta(x))$ is cocompact, so it is a crystallographic group. Finally, the last assertion in point  (v) follows by Lemma \ref{lemma-virtually-abelian}.\\	
Finally, assume that $x'$ is a point of $X$, say  $x'=gx$ with $d(x,x_0) \leq D_0$.
Observe that 
$\overline{\Gamma}_{\varepsilon_i}(x') = g \overline{\Gamma}_{\varepsilon_i}(x)g^{-1}$, so the rank does not change and 
the conditions (iv) and (v) continue to hold  since 
the splitting is $\Gamma$-invariant.
Moreover
$\partial \overline{\Gamma}_{\varepsilon_i}(x')
%=\partial \left( g\overline{\Gamma}_{\lambda_i}(x)g^{-1} \right)
=g \cdot \partial \overline{\Gamma}_{\varepsilon_i}(x)
=g \cdot \partial A = A$, because $\partial A$ is $\Gamma$-invariant, so (iii) still holds. 
Finally, if the group $p_Y(\overline{\Gamma}_{\varepsilon_i}(x))$ preserves  $\lbrace y \rbrace \times \mathbb{R}^k$, then $p_Y(\overline{\Gamma}_{\varepsilon_i}(x'))$ clearly preserves $\lbrace gy \rbrace \times \mathbb{R}^k$, which proves (iv).
 \end{proof}

\vspace{2mm}
\section{The finiteness Theorems}\label{finiteness}
 The goal of this section is to prove Theorem  \ref{theo-intro-finiteness} and its corollaries.
 \linebreak
 %		\ref{cor-intro-finiteness-of-groups}, 	\ref{cor-intro-finiteness-of-manifolds} and		\ref{cor-intro-order-elements}.}
	The work is divided into two steps: we first prove the renormalization Theorem \ref{theo-bound-systole}, from which we  immediately deduce the finiteness up to group isomorphism, and Corollary   \ref{cor-intro-order-elements}. Then, we will improve the result showing the finiteness of the class ${\mathcal O}\text{-CAT}_0(P_0,r_0,D_0)$ (and of ${\mathcal L}\text{-CAT}_0(P_0,r_0,D_0)$ as a particular case) up to equivariant homotopy equivalence of orbispaces, that is Corollary \ref{cor-intro-finiteness-of-groups}). Finally,  we will deduce Corollary \ref{cor-intro-finiteness-of-manifolds} from  the renormalization Theorem   \ref{theo-bound-systole} combined with Cheeger's and Fukaya's finiteness theorems.
 
\subsection{Finiteness up to  group isomorphism}
\label{sec-finiteness-groups}${}$
%Here we prove the following, weak  version of the finiteness Theorem  \ref{theo-intro-finiteness}:
%\begin{theo}
%		\label{theo-weak-finiteness}
%Given $P_0,r_0,D_0 > 0$, there are only finitely many groups 
%acting faithfully, discretely, nonsingularly and $D_0$-cocompactly by isometries on some proper, geodesically complete, $(P_0,r_0)$-packed $\textup{CAT}(0)$-space  (up to isomorphism of abstract groups).  
%\end{theo}
%{\color{cyan} In order to prevent ambiguity we remark that the statement above does not say that the groups acting on a {\emph{specific, fixed}} space are finite up to isomorphism, but that the class of groups acting on {\emph{at least}} one of the  spaces above is finite up to isomorphism.}
%\vspace{1mm}

\noindent Recall that a {\em marked group} is a group $\Gamma$ endowed with a   generating set $\Sigma$.
Two marked groups  $(\Gamma, \Sigma)$ and $(\Gamma', \Sigma')$ are {\em equivalent} if there exists a group isomorphism   $\phi: \Gamma \rightarrow \Gamma'$ such that $\phi(\Sigma)=\Sigma'$; notice that  such a $\phi$ induces an isometry between the respective Cayley graphs.\\
The first step for Theorem \ref{theo-intro-finiteness} is the following:
\begin{prop}
\label{prop-marked} Let $P_0,r_0, D_0$ be given.  For every  fixed  $D \geq D_0$ and $s>0$, there exist only finitely many marked groups $(\Gamma, \Sigma)$ where:
 \begin{itemize}[leftmargin=6mm] 
\item[--] 	 $\Gamma$ is a discrete, $D_0$-cocompact isometry group of  a proper, geodesically complete, $(P_0,r_0)$-packed, \textup{CAT}$(0)$-space  $X$ satisfying  $\textup{sys} (\Gamma,x) \geq s$ for some $x \in X$,   
\item[--]   $\Sigma=\overline{\Sigma}_{2D}(x)$ is a $2D$-short generating set of $\Gamma$ at $x$.
 \end{itemize}
(Finiteness here is meant up to equivalence of marked groups.)
\end{prop}
 
\begin{proof}
	Recall from Section \ref{subsection-isometries} that for any $D\geq D_0$ the group $\Gamma$ admits a presentation as  $\Gamma = \langle  \overline{\Sigma}_{2D}(x) \hspace{1mm} |  \hspace{1mm}{\mathcal R}_{2D}(x)  \rangle$,  where ${\mathcal R}_{2D}(x)$ is a subset of words of length $3$ on the alphabet $\overline{\Sigma}_{2D}(x)$. Therefore the number of equivalence classes of such marked groups $(\Gamma, \Sigma)$, with  $\Sigma=\overline{\Sigma}_{2D}(x)$, is bounded above if  we are able to bound uniformly the cardinality of $\overline{\Sigma}_{2D}(x)$.
%(cp. \cite[Lemma 2.3]{BCGS21}). 
But this is an immediate consequence of  Proposition \ref{prop-packing}; actually, using the fact that the points in the orbit of $x$ are $\frac{s}{2}$-separated, we get
 \vspace{-3mm}
 
 	$$\#\overline{\Sigma}_{2D}(x) \leq \text{Pack}\left(2D, \frac{s}{4}\right) \leq P_0(1+P_0)^{\frac{24D}{s} - 1}.\qedhere$$
\end{proof}

 Now, the main idea to prove Theorem \ref{theo-intro-finiteness} is to use the Splitting Theorem \ref{theo-splitting-weak} to show that, up to increasing in a controlled way the codiameter, we can always suppose that the {\em free-systole} is bounded away from zero by a universal positive constant: this is the content of the renormalization Theorem  \ref{theo-bound-systole}, which is proved below. 
 The proof will show that the space $X'$ is isometric to $X$ (though generally the quotients $\Gamma \backslash X$ and $\Gamma \backslash X'$ are not isometric, since the first one can be $\varepsilon$-collapsed for arbitrarily small $\varepsilon$, while the second one is not collapsed, by construction).\\
Naively, one can think that if $\textup{sys}^\diamond(\Gamma,X)$ is too small then, as we know that $X$ splits $\Gamma$-invariantly as $Y\times \mathbb{R}^k$ by Theorem \ref{theo-splitting-weak}, the new space is  $X' = Y \times s\cdot \mathbb{R}^k$, obtained by dilating the Euclidean factor of a suitable $s>0$,  and the action of $\Gamma$ is the natural one induced on it.
The construction of $X'$ is however a bit more complicate, since this na\"if renormalization is not sufficient in general to enlarge the free systole while keeping the diameter bounded. In order to make things work we need to take into account that $X$ can be collapsed on different subsets at different scales, and an algorithm allowing us to detect them. 
 We refer to Remark \ref{rmk-splitting} for a more precise statement.
\vspace{2mm}

\noindent Recall the constants $n_0=P_0/2$, which bounds the dimension of every proper, geodesically complete,  CAT$(0)$-space  $X$ which is $(P_0, r_0)$-packed, and  $J_0=J_0(P_0)$, introduced at the beginning of Section \ref{sec-splitting}. Also recall  the Margulis constant $\varepsilon_0=\varepsilon_0(P_0, r_0)$ given by Proposition \ref{prop-Margulis-nilpotent}, which we will always assume smaller than $1$ in the sequel.

\begin{proof}[Proof of Theorem \ref{theo-bound-systole}] 
Recall the function $ \sigma_{P_0,r_0,D_0}(\varepsilon)$ of  Theorem \ref{theo-splitting-weak}.
Then we define inductively $D_1 = 2D_0 + \sqrt{n_0}$, $\sigma_1 = \sigma_{P_0,r_0,D_1}( \frac{\varepsilon_0}{4J_0})$ and
$$D_{j} = 2D_{j-1} + \sqrt{n_0} \;\,,\hspace{1cm}
\sigma_{j} = \sigma_{P_0,r_0,D_j} (\sigma_{j-1}) > 0.$$
%$\noindent (i.e. the constant given by Theorem \ref{theo-splitting-weak} for the corresponding parameters).
\noindent We claim that $\Delta_0 := D_{n_0 -1}$ and $s_0 := \sigma_{n_0}$ satisfy the thesis; notice that both depend only on $P_0,r_0$ and $D_0$.\\
We now describe a process which takes the CAT$(0)$-space $X_0:=X$ and produces a new proper, geodesically complete, $(P_0,r_0)$-packed, CAT$(0)$-space $X_1$ on which $\Gamma$ still acts  faithfully and discretely by isometries, still satisfying all the assumptions of the theorem, except that  $\text{diam}(\Gamma \backslash X_1) \leq D_1$; and we will show that,  repeating again and again this process, we end up  with a CAT$(0)$-space $X_j$ with $\textup{sys}^\diamond(\Gamma,X_j) > \sigma_{n_0}$, for some $j\leq n_0$.\\
If $\textup{sys}^\diamond(\Gamma,X_0) > \sigma_{n_0}$, there is nothing to do, and we just set $X'=X_0$.\\
% the algorithm stops and we set $X'=X_1$. 
%It satisfies all the required properties. 
Otherwise,	 $\textup{sys}^\diamond(\Gamma,X) \leq \sigma_{n_0} < \sigma_1 = \sigma_{P_0,r_0,D_0} (\frac{\varepsilon_0}{4J_0}) $, and we apply 
Theorem \ref{theo-splitting-weak} 
%the Proposition  \ref{prop-commensurated}  
with $\varepsilon=\frac{\varepsilon_0}{4J_0}$.
Then, there exists $\varepsilon^\ast_0:=\varepsilon^\ast \in (\sigma_1 ,\frac{\varepsilon_0}{4J_0})$ such that the groups 
$\overline{\Gamma}_{\varepsilon^\ast_0}(x,X_0)$ have all rank $k_0 \geq 1$ for every  $x\in X_0$. We then fix   $x_0 \in X_0$.   
By Proposition  \ref{prop-commensurated} there exists
   a free abelian, finite index  subgroup $A_0$ of $ \overline{\Gamma}_{\varepsilon^\ast_0}(x_0,X)$ of rank $k_0 \geq 1$, which is commensurated in $\Gamma$; and we have that $X_0=\textup{Bd-Min}(\partial A_0)$ splits isometrically and $\Gamma$-invariantly as $Y_0 \times \mathbb{R}^{k_0}$ by Proposition    \ref{prop-commensurated-splitting}.
%The space $X$ then splits isometrically and $\Gamma$-invariantly as $Y \times \mathbb{R}^{k}$ with $k\geq 1$,  there exists $\eta \in (\lambda_1 ,\lambda_0)$ such that the rank of $\overline{\Gamma}_{\eta}(x, X)$ is exactly $k$ for all $x\in X$,  and  the trace at infinity $\partial \overline{\Gamma}_{\eta}(x, X)$ coincides with the boundary at infinity of all sets $\lbrace y \rbrace \times \mathbb{R}^{k}$. 
Moreover,  always by Theorem \ref{theo-splitting-weak}, there exists $y_0\in Y_0$ such that $\overline{\Gamma}_{\varepsilon^\ast_0}(x_0, X_0)$ preserves $\lbrace y_0 \rbrace \times \mathbb{R}^{k_0}$,  and  there exists a subset of $\overline{\Sigma}_{4J_0\cdot \varepsilon^\ast_0}(x_0, X_0)$ whose projection 
%generates a subgroup $L_{\eta_1}(x_1, X_1)$ whose projection 
on $\textup{Isom}(\mathbb{R}^{k_0})$ generates the maximal lattice $\mathcal{L}_{\varepsilon^\ast_0}(x_0,X_0)$ of the crystallographic group $p_{\mathbb{R}^{k_0}}(\overline{\Gamma}_{\varepsilon^\ast_0}(x_0, X_0))$. 
So, we can find
%Now, choose $\bar x\in X$ and 
a shortest basis $\mathcal{B}_0 = \lbrace \bf b^0_1, \ldots, \bf b^0_{k_0} \rbrace$  of the lattice $\mathcal{L}_{\varepsilon^\ast_0}(x_0,X_0)$ whose vectors have all length at most $4J_0\cdot \varepsilon^\ast_0 < \varepsilon_0 < 1$; without loss of generality, we may suppose that $\Vert {\bf b^0_1} \Vert_0 \leq \ldots \leq  \Vert {\bf b^0_{k_0}} \Vert_0 = \lambda \left(\mathcal{L}_{\varepsilon^\ast_0}(x_0,X_0)\right) =:\ell_0 < 1$, where $\Vert \hspace{2mm}\Vert_0$ denotes the Euclidean norm of $\mathbb{R}^{k_0}$.
%since  $\mathcal{L}_{\eta}(x_0,X)$ is generated by the projection of a subset of $\overline{\Sigma}_{4J_0\cdot\eta_1}(x_0, X)$, hence 
% vectors have all length at most $4J_0\cdot \eta < \varepsilon_0 < 1$, and $\mathcal{B}_1$ is a shortest basis.
By \eqref{eq-lattice-relation}, we also know that the covering radius of $\mathcal{L}_{\varepsilon^\ast_0}(x_0,X_0)$ is at most $\frac{\sqrt{k_0}}{2}\cdot \ell_0 \leq \frac{\sqrt{n_0}}{2}\cdot\ell_0 $.\\	
Now, we define the metric space 
$X_1 := Y_0 \times\left( \frac{1}{\ell_0}\cdot \mathbb{R}^{k_0} \right)$. This is again a proper, geodesically complete, CAT$(0)$-space, still $(P_0,r_0)$-packed,   on which $\Gamma$ still acts discretely by isometries (because the splitting of $X_0$ is $\Gamma$-invariant).\linebreak 
	We claim that the action of $\Gamma$ on $X_1$ is $D_1$-cocompact. In fact, let $x = (y,{\bf v})$ be a point of $X_1$. Since the action of $\Gamma$ on $X_0$ is $D_0$-cocompact, we know that there exists $g\in \Gamma$ such that $d_{X_0}\left(x , g\cdot(y_0,{\bf O}) \right) \leq D_0$; moreover, 
	%$d_{X}((y,{\bf v}), g(\bar y, {\bf O})) \leq D_1$, 
as $\overline{\Gamma}_{\varepsilon^\ast_0}(x_0, X)$ preserves $\lbrace  y_0 \rbrace \times \mathbb{R}^{k_0}$, we can compose $g$ with elements of this group in order to find $g' = (g'_1,g'_2) \in \Gamma$ such that $d_{X_0}(x, g' \cdot  (y_0,{\bf O})) \leq D_0$ and 
%$d_{\mathbb{R}^{k_0}}({\bf v}, g'_2 \cdot {\bf O}) \leq  \frac{\sqrt{n_0}}{2}\cdot\ell $.
$\Vert {\bf v} -  g'_2 \cdot {\bf O} \Vert_0 \leq  \frac{\sqrt{n_0}}{2}\cdot\ell_0 $. Therefore
	\begin{equation*}
		\begin{aligned}
	d_{X_1}((x, g' \cdot (y_0,{\bf O})) &
\leq \sqrt{d_{Y_0}(y, g_1' \cdot y_0)^2 + \left(\frac{1}{\ell_{0}}\right)^2\cdot 
%d_{\mathbb{R}^{k_0}}({\bf v}, g'_2 \cdot {\bf O})^2} \\
\Vert {\bf v} -  g'_2 \cdot {\bf O} \Vert_0^2} \\ 
&\leq \sqrt{D_0^2 + \frac{n_0}{4}} \leq D_0 + \frac{\sqrt{n_0}}{2} = \frac{D_1}{2}.
		\end{aligned}
	\end{equation*}
	As   $(y,{\bf v}) \in X_1$ was arbitrary, we then deduce that 
	$ X_1 = \Gamma\cdot \overline{B}_{X_1}\left(( y_0,{\bf O}), \frac{D_1}{2}\right),$ 
	so $\Gamma < \text{Isom}(X_1)$ is $D_1$-cocompact. 
	If now $\textup{sys}^\diamond(\Gamma,X_1) > \sigma_{n_0}$, we stop the process and  set $X' = X_1$: this space has all the desired properties. \\
	Otherwise, we have
$\textup{sys}^\diamond(\Gamma,X_1) \leq   \sigma_{n_0} 
 < \sigma_2 = \sigma_{P_0,r_0,D_1}  (\sigma_1)$
	 and we can  apply again Theorem \ref{theo-splitting-weak},
 Proposition  \ref{prop-commensurated} and Proposition    \ref{prop-commensurated-splitting} 
to $X_1$, with  $\varepsilon=\sigma_1$. \linebreak 
   Then, there exists  $\varepsilon^\ast_1 \in (\sigma_2 ,\sigma_1)$ 
   such that the groups 
$\overline{\Gamma}_{\varepsilon^\ast_1}(x,X_1)$ have rank $k_1 \geq 1$ for every  $x\in X_1$, in particular  $\text{rk}(\overline{\Gamma}_{\varepsilon^\ast_1}(x_0,X_1))=k_1$.  Moreover, there exists
   a free abelian, finite index  subgroup $A_1 <  \overline{\Gamma}_{\varepsilon^\ast_1}(x_0,X_1)$ of rank $k_1$  \linebreak which is commensurated in $\Gamma$, the space $X_1=\textup{Bd-Min}(\partial A_1)$ splits isometrically and $\Gamma$-invariantly as $Y_1 \times \mathbb{R}^{k_1}$, and the trace at infinity $\partial \overline{\Gamma}_{\varepsilon^\ast_1}(x_0, X_1)$ coincides with the boundary at infinity of all the sets $\lbrace y \rbrace \times \mathbb{R}^{k_1}$; and  there exists a subset of $\overline{\Sigma}_{4J_0\cdot \varepsilon^\ast_1}(x_0, X_1)$ whose projection  on $\textup{Isom}(\mathbb{R}^{k_1})$ generates the maximal lattice $\mathcal{L}_{\varepsilon^\ast_1}(x_0,X_1)$ of   $p_{\mathbb{R}^{k_1}}(\overline{\Gamma}_{\varepsilon^\ast_1}(x_0, X_1))$. 
   Therefore, we can find
a shortest basis $\mathcal{B}^1 = \lbrace \bf b_1^1, \ldots, \bf b_{k_1}^1 \rbrace$ of $\mathcal{L}_{\varepsilon^\ast_1}(x_0,X_1)$ with lengths (with respect to the Euclidean norm  $\Vert \hspace{2mm} \Vert_1$ of $\mathbb{R}^{k_1}$)
$$\Vert {\bf b_{1}^1}\Vert_1 \leq \ldots \leq  \Vert {\bf b_{k_1}^1} \Vert_1 = \lambda \left(\mathcal{L}_{\varepsilon^\ast_1}(x_0,X_1)\right) =:\ell_1 \leq 4J_0\cdot \sigma_1 <  \varepsilon_0 < 1.$$
 Observe that, as
  by construction the factor $\left(\frac{1}{\ell_0 }\right)^2$ is bigger than $1$, we have	\begin{equation}
		\label{eq-group-inclusion}
		\overline{\Gamma}_{\varepsilon^\ast_1}( x_0, X_1) < \overline{\Gamma}_{\sigma_1}( x_0, X_1) < \overline{\Gamma}_{\sigma_1}( x_0, X_0) < \overline{\Gamma}_{\varepsilon^\ast_0}( x_0, X_0),
	\end{equation}
	so
	$k_1= \text{rk}(\overline{\Gamma}_{\varepsilon^\ast_1}( x_0, X_1)) \leq k_0 = \text{rk}(\overline{\Gamma}_{\varepsilon^\ast_0}( x_0, X_0))$, 
	and $A_1 < A_0$. In particular, the metric splittings of $X_1$ as $Y_0 \times \left( \frac{1}{\ell_0}{\mathbb R}^{k_0}\right)$  and $Y_1 \times {\mathbb R}^{k_1}$  determined, respectively, by the groups $A_0$ and $A_1$ satisfy ${\mathbb R}^{k_1} \subset \frac{1}{\ell_0} {\mathbb R}^{k_0}$ and $Y_0 \subset Y_1$, because of  the second part of Proposition \ref{prop-commensurated-splitting}. \\
	We will now show that $k_1 \lneq k_0$.
 Actually, suppose that  $k_1 = k_0=:k$. 
 Then, 
%By Lemma \ref{lemma-trace-infinity-inclusion} 
%by Proposition	\ref{prop-trace-infinity}  
%the groups $\overline{\Gamma}_{\eta_1}( x_0, X_1)$ and $\overline{\Gamma}_{\eta_0}( x_0, X)$
the groups $A_0$ and $A_1$
%{\color{red}acting on $X_1$,}  
%have the same trace at infinity
%{\color{red}and, by Proposition \ref{prop-commensurated-splitting},
split the same Euclidean factor and  $Y_0= Y_1$.
%{\color{blue} ma come sono messi $Y_1$ e $Y$? non serve sapere che $Y_1 \subset Y$ per concludere?}\\
%We also now that there exists a subset of $\overline{\Sigma}_{4J_0\cdot\eta_1}(x_0, X_1)$ whose projection  on $\textup{Isom}(\mathbb{R}^{k})$ generates the maximal lattice $\mathcal{L}_{\eta_1}(x_0,X_1)$ of   $p_{\mathbb{R}^{k_1}}(\overline{\Gamma}_{\eta_1}(x_0, X_1))$, therefore we can find a shortest basis $\mathcal{B}^1 = \lbrace \bf b_1^1, \ldots, \bf b_k^1 \rbrace$ of $\mathcal{L}_{\eta_1}(x_0,X_1)$ with lengths (with respect to the Euclidean norm  $\Vert \hspace{2mm} \Vert_1$ of $\frac{1}{\ell_0}\cdot \mathbb{R}^{k}$)
%$$  \Vert {\bf b_{i}^1} \Vert_1 
%%= {\color{red}\ell_{X_1}(b_{i}^1)} \leq 4J_0\cdot\eta_1 
%< {\color{red} 4J_0\cdot \sigma_0 <  \varepsilon_0 }< 1, \text{ for all } i=1,...,k.$$
The lengths of the basis $\mathcal{B}^1 $ with respect to the Euclidean norm $\Vert \hspace{2mm} \Vert_0$ of  the Euclidean factor ${\mathbb R}^{k_0}$ of $X_0$ are $\Vert {\bf b_{i}^1} \Vert_0 = \ell_0  \cdot \Vert {\bf b_{i}^1} \Vert_1 < \ell_0$ 
%\Vert_1 < \ell_0 
%	\begin{equation}
%		\label{eq-short-basis}
%		\Vert b_{i}^1 \Vert_0 = \ell_0  \cdot \Vert b_{i}^1 
%\Vert_1 < \ell_0 
%	\end{equation}
	for every $i=1,\ldots,k$.
But then, since $\mathcal{L}_{\varepsilon^\ast_1}(x_0,X_1) < \mathcal{L}_{\varepsilon^\ast_0}(x_0,X_0)$ by \eqref{eq-group-inclusion}, we would be able to find $k$ independent vectors of $\mathcal{L}_{\varepsilon^\ast_0}(x_0,X_0)$ of length less than its shortest generating radius, which is impossible.
% A subset $\lbrace b_{1,2},\ldots,b_{k_1,2}\rbrace$ of $\overline{\Sigma}_{4J_0\cdot\eta_2}(\bar x, X_2)$ generates a subgroup $L_{\eta_2}(\bar x, X_2)$ whose projection on $\textup{Isom}(\mathbb{R}^{k_1})$ is the maximal lattice $\mathcal{L}_{\eta_2}(\bar x,X_2)$ of the crystallographic group $p_{\mathbb{R}^{k_1}}(\overline{\Gamma}_{\eta_2}(\bar x, X_2	
%	Call ${\bf b_{i,2}} = p_{\mathbb{R}^{k_1}}(b_{i,2})$ for $i=1,\ldots,k_1$. They are translation whose length is
%	$$\Vert {\bf b_{i,2}} \Vert = \ell_{X_2}(b_{i,2}) \leq 4J_0\cdot\eta_2 < 4J_0\cdot s_1 < 4J_0\cdot \varepsilon_0 < 1.$$
%	Thus
%	\begin{equation}
%		\label{eq-short-basis}
%		\ell_{X_1}(b_{i,2}) = \ell_1 \cdot \ell_{X_2}(b_{i,2}) < \ell_1
%	\end{equation}
%	for every $i=1,\ldots,k_1$. By \eqref{eq-group-inclusion} we have  $\mathcal{L}_{\eta_2}(\bar x,X_2) < \mathcal{L}_{\eta_1}(\bar x,X_1)$ and by \eqref{eq-short-basis} we are able to find $k_1$ independent vectors of $\mathcal{L}_{\eta_1}(\bar x,X_1)$ of length less than its shortest generating radius, which is impossible.\\
	This shows that $k_1 < k_0$.\\
We can now define   $X_2:= Y_1 \! \times \left(\frac{1}{\ell_1} \!\cdot \mathbb{R}^{k_1} \right)$, on which $\Gamma$ acts faithfully, discretely by isometries,  $D_2$-cocompactly, for $D_2 = 2D_1 + \sqrt{n_0}$ computed as before. 
We can repeat this process to get a sequence of proper, geodesically complete, $(P_0,r_0)$-packed, CAT$(0)$-spaces $X_j$ on which $\Gamma$ always acts faithfully, discretely and $D_j$-cocompactly by isometries. Moreover at each step 
%for $j=1,\ldots,n_0+1$ 
either $\textup{sys}^\diamond(\Gamma,X_j) \geq \sigma_{n_0}$ or the $\varepsilon^\ast_j$-splitting rank $k_j$ of $X_j$ provided by Theorem \ref{theo-splitting-weak} is strictly smaller than the $\varepsilon^\ast_{j-1}$-splitting rank $k_{j-1}$ of $X_{j-1}$. Since the splitting rank of $X_0$ is at most $n_0$ there must exist $j \in \lbrace 1,\ldots, n_0 \rbrace$ such that $\textup{sys}^\diamond(\Gamma,X_j) \geq \sigma_{n_0}$. The proof then ends by setting $X' = X_j$.\\
Clearly, we may take, explicitely,  $\Delta_0 = 2^{n_0}D_0 + (2^{n_0} - 1)\sqrt{n_0}$.\\
	 Finally, observe that, by construction, $X'$ is isometric to the initial space $X$, and that the action of $\Gamma$ on $X_j$ is nonsingular if and only if the action of $\Gamma$ on $X_{j-1}$, and by induction on $X$, was nonsingular.
\end{proof}
 
\begin{obs} 
	\label{rmk-splitting}
%	We add here a couple of remark about the proof.
%	\begin{itemize}
%		\item[(i)] One could try to take $\tilde{s}_j = s(P_0,r_0,\Delta_0,\frac{\varepsilon_0}{4})$ constant for any $j$. Indeed all the arguments work with this choice except maybe for the proof of \eqref{eq-group-inclusion}. Indeed in our proof the inclusion is clear (because $\eta_2 < s_1 \leq \eta_1$), while with the choice of $\tilde{s}_j$ we would have to prove that each element of $\overline{\Gamma}_{\eta_2}(\bar x, X_2)$ acts on $X_1$ by fixing a point in $Y_1$. This is not completely trivial since without \eqref{eq-group-inclusion} it is not obvious that the two splittings of $X_1$ and $X_2$ are related. 
%\item[(ii)] 
		The proof of Theorem \ref{theo-bound-systole} actually gives us something more. Indeed we  produce a sequence of free abelian commensurated subgroups of $\Gamma$
		$$\lbrace \textup{id} \rbrace \lneqq A_0 \lneqq A_1\ldots \lneqq A_m$$
		for some $m \leq n_0-1$, such that:
		\begin{itemize}
			\item[(a)] denoting by $k_j$ the rank of $A_j$, then $1 \leq  k_0 <k_1 \ldots < k_m$;
			\item[(b)] setting $h_j = k_j - k_{j-1}$ then there is a corresponding  isometric and $\Gamma$-invariant splitting of $X$ as $Y\times \mathbb{R}^{h_0} \times \mathbb{R}^{h_1} \times \cdots \times \mathbb{R}^{h_m}$. Moreover, there exist $0<L_j < 1$ such that the natural action of $\Gamma$ on the space
\begin{equation}
\label{eq-finalsplitting}
X' := Y\times \left(\frac{1}{L_0}\cdot\mathbb{R}^{h_0} \right) \times \left(\frac{1}{L_1}\cdot\mathbb{R}^{h_1} \right)\times \cdots \times \left( \frac{1}{L_m}\cdot\mathbb{R}^{h_m} \right)
\end{equation}
			is $\Delta_0$-cocompact with free-systole at least $s_0$. 
			%Observe that the space $X'$ is isometric to $X$.
		\end{itemize}
\end{obs}

The finiteness Theorem \ref{theo-intro-finiteness}  is now an immediate consequence of the above renormalization result,  combined with Proposition \ref{prop-marked} and Theorem \ref{prop-vol-sys-dias}:

\begin{proof}[Proof of Theorem \ref{theo-intro-finiteness}]
	By Theorem \ref{theo-bound-systole} every group $\Gamma$ under consideration acts discretely by isometries and $\Delta_0$-cocompactly on a proper, geodesically complete, $(P_0,r_0)$-packed, CAT$(0)$-space $X'$ with $\textup{sys}^\diamond(\Gamma,X') \geq s_0$.
	The action is still nonsingular, so by Theorem \ref{prop-vol-sys-dias} we can find $x'\in X'$ with $\textup{sys}(\Gamma,x') \geq b_0 \cdot s_0$,   a positive lower bound depending only on $P_0,r_0$ and $D_0$.  \linebreak 
	Then, applying Proposition \ref{prop-marked} with $D_0=\Delta_0=D$ and   $s=b_0\cdot s_0$, we conclude the proof.
%    the theorem is an immediate consequence of the following:
% 	By \cite[Remark III.$\Gamma$.1.7]{BH09} a presentation of $\Gamma$ is given by $\langle \mathcal{S}\vert \mathcal{R}\rangle$, where $\mathcal{S} = \overline{\Sigma}_{3\Delta_0 + 1}(x')$ and $\mathcal{R}$ is the set of reduced words of length at most $10$ in the alphabet $\mathcal{S}$ representing the identity. Therefore the number of isomorphism classes of possible groups $\Gamma$ is bounded provided we are able to bound uniformly the cardinality of $\mathcal{S}$ (cp. \cite[Lemma 2.3]{BCGS21}). 
%	But this is an immediate consequence of  Proposition \ref{prop-packing}; actually, using the fact that the points in the orbit of $x'$ are $(b_0\cdot s_0)$-separated, we get
%	$$\#\overline{\Sigma}_{3\Delta_0 + 1}(x') \leq \text{Pack}\left(3\Delta_0 + 1, b_0\cdot s_0\right) \leq P_0(1+P_0)^{\frac{3(3\Delta_0+1)}{b_0\cdot s_0} - 1}.$$

\vspace{-5mm}
\end{proof}

 \begin{obs}
 	\label{rem_Bass}
 {\em The assumption that the groups $\Gamma$ act nonsingularly is  necessary for   Theorem \ref{theo-intro-finiteness}.
    Actually, in \cite[Theorem 7.1]{BK90},  Bass and  Kulkarni   exhibit  an infinite ascending family of  discrete  groups $\Gamma_1 < \Gamma_2 < \cdots$ acting (singularly) on a regular tree $X$ with bounded valency, with same compact quotient, in particular with diam$(\Gamma_j \backslash X) \leq 2$ for all $j$. Moreover these groups are lattices satisfying Vol$(\Gamma_i \backslash \backslash X) \rightarrow 0$, where the volume of   $\Gamma$ in $X$ is defined as 
$$\text{Vol}(\Gamma  \backslash \backslash X) = \sum_{x \in \Gamma \backslash X} \frac{1}{| \text{Stab}_\Gamma(x) |}. $$
This family contains infinitely many different groups, since the minimal order $\sigma_i$ of a torsion subgroup in  $\Gamma_i$  tends to infinity as the volume goes to zero (every torsion subgroup of $\Gamma_i$ stabilizes some point, as $X$ is CAT$(0)$, hence $\frac{1}{\sigma_i}$  is the dominant term of  the sum yielding $\text{Vol}(\Gamma_i  \backslash \backslash X)$).
 }
\end{obs}
 
\begin{proof}[Proof of Corollary \ref{cor-intro-order-elements}]
	The number of isomorphism types of possible group is finite by Theorem \ref{cor-intro-finiteness-of-groups}, so the thesis is a trivial consequence of this theorem. However we can give a more costructive proof that gives an explicit upper bound for the order of finite subgroups.
	By Theorem \ref{theo-bound-systole} we can suppose that $\Gamma$ is $\Delta_0$-cocompact and $\text{sys}^\diamond(\Gamma,X)\geq s_0$. Let $F< \Gamma$ be a finite subgroup. It fixes a point $x_0 \in X$ (cp. \cite[Corollary II.2.8]{BH09}). Let $\mathrm{D}_{x_0}$ be the Dirichlet domain of $\Gamma$ at $x_0$, which is clearly $F$-invariant. By Theorem \ref{prop-vol-sys-dias} there exists some point $x\in \mathrm{D}_{x_0}$ where  $\text{sys}(\Gamma, x)\geq b_0\cdot s_0 > 0$.  Therefore the orbit $\lbrace fx\rbrace_{f \in F}$ is a $\frac{b_0\cdot s_0}{2}$-separated subset of $\mathrm{D}_{x_0}$ and the balls $B(fx, \frac{b_0\cdot s_0}{4})$ are all disjoint. Using the fact that $\mathrm{D}_{x_0}$ is contained in $\overline{B}(x_0,\Delta_0)$ we get
	$$V(\Delta_0) \geq \mu_X(\mathrm{D}_{x_0}) \geq \# F \cdot v\left(\frac{b_0\cdot s_0}{4}\right),$$
	by Proposition \ref{prop-packing}.(iii).
	This yields the  explicit bound $\#F \leq \frac{V(\Delta_0)}{v\left(\frac{b_0\cdot s_0}{4}\right)}$
	for the cardinality of $F$, only depending only on $P_0,r_0$ and $D_0$.
\end{proof}

\subsection{Finiteness up to equivariant homotopy equivalence}
\label{sec-finiteness-spaces}${}$\\
	We have proved so far that the number of discrete, nonsingular and $D_0$-cocompact groups of isometries $\Gamma$ of proper, geodesically complete, $(P_0,r_0)$-packed, CAT$(0)$-spaces $X$ is finite up to isomorphism of abstract groups. \\
	In this section we will show the finiteness up to {\em equivariant homotopy equivalence of orbispaces}. That is, we can upgrade every 
isomorphism of  groups $\varphi: \Gamma \rightarrow \Gamma'$, acting discretely, nonsingularly and cocompactly by isometries respectively on CAT$(0)$-spaces $X$ and $X'$, to a homotopy equivalence $F: X \rightarrow X'$ which is $\varphi$-equivariant, i.e. $F(g \cdot x) = \varphi(g) \cdot F(x)$ for all $g \in \Gamma$ and all $x \in X$.  Namely we show:
	
\begin{prop}\label{prop-homeq}
		Let $\Gamma$ and $ \Gamma'$ be discrete, cocompact, nonsingular isometry groups of two proper, geodesically complete, $(P_0,r_0)$-packed, $\textup{CAT}(0)$-spaces $X$ and $X'$, respectively. 
		%< \textup{Isom}(X)$ and $\Gamma' < \textup{Isom}(X')$ be discrete, cocompact and non-singular. 
		If there exists a group isomorphism $\varphi: \Gamma \rightarrow \Gamma'$, then there exists a $\varphi$-equivariant homotopy equivalence $f: X \rightarrow X'$. 
	\end{prop}

% 	\begin{prop}\label{prop-homeq}
%		Let $X,X'$ be complete, geodesically complete, $(P_0,r_0)$-packed \textup{CAT}$(0)$-spaces, and let $\Gamma, \Gamma'$ be discrete, cocompact, nonsingular isometry groups of $X$ and $X'$ respectively. 
		%%< \textup{Isom}(X)$ and $\Gamma' < \textup{Isom}(X')$ be discrete, cocompact and non-singular. 
% Assume that there exist a group isomorphism $\varphi: \Gamma \rightarrow \Gamma'$: then, there exists a $\varphi$-equivariant homotopy equivalence $f: X \rightarrow X'$. 
%	\end{prop}

	\noindent  This is a CAT(0) version of a result about the realization of orbifold group-isomorphisms by orbi-maps equivalences, see \cite[Theorem 2.5]{Yam90}. We will give a direct  proof of this fact,  
	using the barycenter technique introduced by Besson-Courtois-Gallot   \cite{BCG95} for strictly negatively curved manifolds (cp.  \cite{Sam99} for an approach similar to the one we follow here).
	\vspace{1mm}
	
	\noindent Let $X$ be a proper, CAT$(0)$-metric space. Consider the space ${\mathcal M}_2(X)$ of positive, finite, Borel measures $\mu$ on $X$ with finite second moment, i.e. such that the distance function $d(x,\cdot)$ is $\mu$-square integrable for some (hence every) $x \in X$.
	The  {\em barycenter}  of $\mu \in {\mathcal{M}_2}(X)$  is defined as the unique point of minimum of the function 
	$$\mathcal{B}_{\mu} (x)= \int_{X} d(x,y)^2 \, d \mu (y).$$
	Notice that $\mathcal{B}_{\mu} (x)$ tends to $+\infty$ for $x \rightarrow \infty$ since for every fixed $x_0$ in $X$ we have, by triangular and Schwarz inequalities,
	$$\mathcal{B}_{\mu} (x) 
	\geq d(x,x_0)^2\mu(X) - 2 d(x,x_0)\mu(X)\mathcal{B}_\mu(x_0) +
	\mathcal{B}_\mu(x_0)$$
	which diverges as $d(x,x_0) \rightarrow + \infty$. Moreover, by standard comparison with the Euclidean space, the function  $d(x,\cdot)^2$ is also strictly convex, namely $2$-convex in the sense of \cite{Kl99}: that is,  
	$d(x,c(t)) -  t^2$ is a convex function, for every geodesic $c: [a,b] \rightarrow X$.
	It follows that the function $\mathcal{B}_{\mu}$ is $2\mu(X)$-convex as well, which implies that $\mathcal{B}_{\mu} $ has a unique point of minimum \linebreak (cp. \cite[Lemma 2.3]{Kl99}). 
	Therefore $\texttt{bar} [ \mu] := \text{arg}\, \text{min}  (\mathcal{B}_{\mu} )$ is well-defined.\\
	It is straightforward to check that the barycenter is equivariant with respect to the natural actions of $\text{Isom}(X)$ on $X$ and  ${\mathcal M}_2(X)$, that is:
	\begin{equation}\label{eq-barequivariant} 
		\texttt{bar}[ g_\ast \mu ] = g \cdot \texttt{bar} [ \mu ], \hspace{2mm}\forall g  \in \text{Isom}(X), \forall \mu \in {\mathcal M}_2(X)
	\end{equation}
We can now give the	Proof of Proposition \ref{prop-homeq}:
	\begin{proof}[Proof of Proposition \ref{prop-homeq}.]
		As $(\Gamma,X)$ and $(\Gamma',X')$ are nonsingular, we can choose $x_0 \in X$ and $x_0' \in X'$ such that the stabilizers $\text{Stab}_\Gamma(x_0)$,  $\text{Stab}_{\Gamma'}(x_0')$ are trivial. 
		Then there  exists a unique $\varphi$-equivariant map $\phi: \Gamma x_0 \rightarrow \Gamma' x_0'$ sending $x_0$ to $x_0'$, namely $\phi (g \cdot x_0)= \varphi( g) \cdot x_0'$.  Now, by $\check{\textup{S}}$varc-Milnor Lemma, the spaces $X$ and $X'$ are respectively quasi-isometric to the orbits $\Gamma x_0$ and $\Gamma'x_0'$ and also to the marked groups $(\Gamma, \Sigma_{2D}(x_0))$ and  $(\Gamma',  \Sigma_{2D}(x_0'))$ endowed with their word metrics.
		Moreover, since the groups $\Gamma$ and $\Gamma'$ are isomorphic, and any two word metrics associated to finite generating sets on the same group are equivalent, we deduce that  the map $\phi$ is an $(a,b)$-quasi-isometry for suitable   $a,b$, in particular
		\begin{equation}\label{eq-qi}
			d(\varphi(g_1) x_0', \varphi (g_2) x_0') \leq a d(g_1 x_0, g_2 x_0)+b \hspace{5mm} \text{ for all } g_1,g_2\in \Gamma
		\end{equation}
%		{\color{red} for all $g_1,g_2\in \Gamma$.}
		We  now define a  $\varphi$-equivariant homotopy equivalence $f: X \rightarrow X'$ as follows:
		choose any $h > \frac{\log(1+P_0)}{r_0}$ (the upper bound of the entropy of $X$ given by Proposition \ref{prop-packing}.(iv) and consider the family of measures $\mu_x \in  {\mathcal  M}_2(X')$, indexed by $x \in X $ and supported by the orbit  of $x_0'$, given by 
		$$ \mu_x :=  \sum_{g\in \Gamma} e^{-hd(x,g x_0)} \delta_{\varphi(g) x_0'},$$
		where $\delta_{\varphi(g)x_0'}$ is the Dirac measure at $\varphi(g)x_0'$,
		and then define $$f (x) := \texttt{bar} [\mu_x].$$
		Notice that the total mass of $\mu_x$ coincides with the value of the Poincar\'e series  $P_\Gamma(x,x_0,s)=\sum_{g\in \Gamma} e^{-sd(x,g x_0)}$ of the group $\Gamma$ acting on $X$ for $s=h$, which is finite since $h$ is chosen greater than the critical exponent of the series (recall that $\text{Ent}(X)$ equals the critical exponent of the group $\Gamma$ acting on $X$, since the action is cocompact, see for instance \cite{BCGS17} or \cite[Proposition 5.7]{Cav21ter}).
	Moreover, the function $d(x_0',\cdot)$ is square summable with respect to $\mu_x$, since we have by (\ref{eq-qi})
			\begin{equation*}
				\begin{aligned}
					\int_{X'} \hspace{-2mm}d(x_0',y)^2 d\mu_x(y) &= \sum_{g \in \Gamma} d(x_0',\varphi(g) x_0')^2 e^{-hd(x,g x_0)}\\ 
					&\leq \sum_{g \in \Gamma} (a\cdot d(x_0,g x_0)+b)^2 e^{-hd(x,g x_0)} 
				\end{aligned}
			\end{equation*}
			which is finite as the Poincar\'e series converges exponentially fast for $s=h$.\\
			Therefore the map $f$ is well-defined.
	\vspace{2mm}

			\noindent {\bf Step 1:}   $f$ is continuous. \\ First, we show that   $\mathcal{B}_{\mu_{x_n}}$ converge uniformly on compacts to $\mathcal{B}_{\mu_{x}} $ for $x_n \rightarrow x $ in $X$. 
			Actually, let  $K \subset X'$ be a fixed compact subset.  Since the action of $\Gamma'$ on $X'$ is discrete, the subset
			$$S = \{ g' \in \Gamma'  \hspace{1mm} |  \hspace{1mm} g' B(x_0', 2D_0) \cap K \neq \emptyset \} $$
			is finite. 
			Then for every  $x' \in K$ choose  some 
			$g_{x'} \in \Gamma$ with $\varphi(g_{x'}) \in S$ such that  $d(x',\varphi(g_{x'}) x_0')\leq 2D_0$. 
			From \eqref{eq-qi}  and the triangular inequality we find    that, for $d(x_n,x)<\varepsilon$, it holds
			$$\left| \mathcal{B}_{\mu_{x_n}}  (x') - \mathcal{B}_{\mu_{x}}   (x')  \right| 
			\leq \sum_{g \in \Gamma} d(x', \varphi(g) \cdot x_0')^2  \left|  e^{-hd(x_n, g \cdot x_0)} - e^{-hd(x, g \cdot x_0)} \right| \hspace{10mm}$$
			$$\hspace{15mm} \leq h \left( \sum_{g \in \Gamma}  \left( d (\varphi(g_{x'}) x_0', \varphi( g) x_0') +2D_0 \right)^2  e^{-h  \left(d(x, g x_0) -\varepsilon\right)} \right) d(x_n,x)$$
			$$  \leq h \left( \sum_{g \in \Gamma}  \left( a \cdot d (x_0,  g  \cdot x_0) +b' \right)^2  e^{-h  \left(d(x,   g_{x'} g x_0) -\varepsilon\right)} \right) d(x_n,x)$$
			$$ \hspace{13mm} \leq h e^{h\varepsilon} e^{hd( x,g_{x'} x)} \left( \sum_{g \in \Gamma}  \left( a\cdot d(x_0, g \cdot  x_0)+b'\right)^2  e^{-h  d(x, g \cdot x_0) } \right) d(x_n,x) $$
			for $b'=b+2D_0$. Now,   the series in parentheses is bounded above independently of $n$ and $x'$, while (for   fixed $x$)  the term $ e^{hd( x,g_{x'} \cdot x)}$ is uniformly bounded on $K$ since $g_{x'}$ belong to the finite subset $\varphi^{-1}(S)$. The above estimate then implies that  $\mathcal{B}_{\mu_{x_n}}\rightarrow \mathcal{B}_{\mu_{x}} $  uniformily on $K$ when $x_n \rightarrow x$. \\
			Secondly, we call $m=\min \mathcal{B}_{\mu_{x}}$ and show that there exists $R$ such that  for every $n \gg 0$ the functions $ \mathcal{B}_{\mu_{x_n}} $ are greater than $2m$ on  $X' \setminus \overline{ B}(\texttt{bar}[\mu_x],R)$.\\
			Actually, let $R$ be  such that   $ \mathcal{B}_{\mu_{x}} >2m$ outside  $B(\texttt{bar}[\mu_x],R/2)$ (recall that $\mathcal{B}_{\mu_{x}} $ is proper, since we showed that  
			$\lim_{x' \rightarrow \infty }\mathcal{B}_{\mu} (x') = +\infty$ for all $\mu \in {\mathcal M}_2(X)$), and assume that  we have infinitely many points $x_n'$ with $d(\texttt{bar}[\mu_x], x_n') >R$ such that $ \mathcal{B}_{\mu_{x_n}} (x_n') \leq 2m$. Then, calling $y'_n \in [ \texttt{bar}[\mu_x],x'_n]$ the point at distance $R$ from $\texttt{bar}[\mu_x]$, 
we would have, by convexity,
			$$    \mathcal{B}_{\mu_{x_n}} (y'_n)  
			\leq \max \lbrace \mathcal{B}_{\mu_{x_n}}( \texttt{bar}[\mu_x]  ), 2m  \rbrace.$$
			Since $ \mathcal{B}_{\mu_{x_n}} \rightarrow \mathcal{B}_{\mu_{x}}$ uniformly on   $\overline{ B}(\texttt{bar}[\mu_x],R)$ we would deduce that \linebreak
			$\mathcal{B}_{\mu_{x}} (y_n')  \leq  2m$, which contradicts the choice of $R$, since $d( \texttt{bar}[\mu_x], y_n') > R/2$.\\
			Now, since $ \mathcal{B}_{\mu_{x_n}}  > 2m$ on  $X' \setminus \overline{B}(\texttt{bar}[\mu_x],R)$ for all $n \gg 0$,  the uniform convergence  
			$\mathcal{B}_{\mu_{x_n}}  \rightarrow \mathcal{B}_{\mu_{x}}$  
	on  $\overline{B}(\texttt{bar}[\mu_x],R) $ implies that the sequence  \linebreak $f(x_n)=\texttt{bar}[\mu_{x_n}]$  of  (unique) minimum points   of $\mathcal{B}_{\mu_{x_n}} $   converge to  the (unique) minimum point $f(x)=\texttt{bar}[\mu_{x}]$ of  $\mathcal{B}_{\mu_{x}}$. In fact, we have that $\texttt{bar}[\mu_{x_n}] \in  \overline{B}(\texttt{bar}[\mu_x],R)$  for all $n \gg 0$, and
			\begin{equation}
				\label{eq-lim}
				\mathcal{B}_{\mu_{x_n}} (\texttt{bar}[\mu_x]) \geq  \mathcal{B}_{\mu_{x_n}} (\texttt{bar}[\mu_{x_n}])
			\end{equation}  
			so  if  $\texttt{bar}[\mu_{x_n}]$ accumulates to a point $b_\infty$, passing to the limit in \eqref{eq-lim} yields  $\mathcal{B}_{\mu_{x}} (\texttt{bar}[\mu_x]) \geq  \mathcal{B}_{\mu_{x}} (b_\infty)$. This implies that $b_\infty=\texttt{bar}[\mu_x]$, by unicity of the minimum point of $\mathcal{B}_{\mu_{x}}$.
\vspace{2mm}			
			% Hence $f(x_n)=\text{bar}[\mu_{x_n}]$ converge to $f(x)=\text{bar}[\mu_{x}]$.
			
			\noindent {\bf  Step 2:}   $f$ is a $\varphi$-equivariant homotopy equivalence.	\\
				Firstly, the  map $f$ is $\varphi$-equivariant, since for all  $g  \in \Gamma$ we have
			$$\varphi(g)_{\! \ast} \, \mu_x 
			%= \sum_{g\in \Gamma} e^{-h d(x,gx_0)}  \varphi(\gamma)_\ast \delta_{\varphi(g)x_0'}
			= \sum_{\gamma\in \Gamma} e^{-h d(x,\gamma x_0)}   \delta_{ \varphi(g \gamma ) \cdot x_0'}
			= \sum_{\gamma \in \Gamma} e^{-h d(x,g^{-1} \gamma x_0)}    \delta_{  \varphi(\gamma )x_0'}
			= \mu_{g x}$$
			and   by \eqref{eq-barequivariant} we deduce
			$f (g \cdot x) 
			%= bar [\mu_{\gamma \cdot x}] 
			= \texttt{bar} [\varphi(g)_\ast \mu_{x}] = \varphi(g) \cdot  \texttt{bar}[\mu_{  x}] = \varphi(g) \cdot f ( x) $.\\
			%We are left to show that $f: X \rightarrow X'$ is a $\varphi$-equivariant homotopy equivalence. 
			Then, from the inverse  map $\phi^{-1}: \Gamma' x_0' \rightarrow  \Gamma x_0$
			% which is $\varphi^{-1}$-equivariant,  
			we can analogously construct a  $\varphi^{-1}$-equivariant, continuous map $f': X' \rightarrow X$.
			Now consider the homotopy map $H:X \times [0,1] \rightarrow X$ defined by the formula $H(x,t)=tx + (1-t) \,f' \! \circ \! f(x)$ (where  $tx + (1-t)y$ denotes the point   on the geodesic segment $[x,y]$ at distance $t/d(x,y)$ from $x$).				The composition $f' \circ f$ is $\Gamma$-equivariant, so we  deduce that $H$ is a $\Gamma$-equivariant  homotopy between $f' \circ f$  and $\text{id}_X$.  \\
			Similarly, one proves that the map   $H'(x',t)=tx' + (1-t) \,f \! \circ \! f'(x)$ is a $\Gamma'$-equivariant homotopy between $f  \circ   f'$ and $\text{id}_{X'}$.
		\end{proof}
	
	\vspace{1mm}
\subsection{Finiteness of nonpositively curved orbifolds}
\label{sec-finiteness-manifolds}${}$\\
%We used, so far, the term CAT$(0)$-orbispace as in \cite{Fuk86}, as a synonym of  {\em quotient} $M=\Gamma \backslash X$ {\em of a (proper, geodesically complete) \textup{CAT}$(0)$-space.} \linebreak
%%(resp. a CAT$(0)$-homology manifold) {\em $X$ by a discrete, isometry group $\Gamma$.} 
%Here we will restrict our attention to  the case where $X$ is a Hadamard manifold (a complete, connected and simply connected, nonpositively curved Riemannian manifold)
%Therefore, for us  a {\em nonpositively curved Riemannian orbifold} with sectional curvature $-\kappa^2 \leq k(M) \leq 0$  is the quotient of a Hadamard manifold satisfying the same curvature bounds by a discrete isometry group. 
We restrict here our attention to CAT$(0)$-orbispaces which are quotients, by a discrete isometry group $\Gamma$, of a   {\em Hadamard manifold}  $X$   (that is,  a complete, connected and simply connected, nonpositively curved Riemannian manifold) with pinched sectional curvature $-\kappa^2 \leq k(X) \leq 0$: we call such a quotient  $M= \Gamma \backslash X$  a   {\em nonpositively curved Riemannian orbifold} with curvature $k(M) \geq -\kappa^2$.
Recall that, in this case, %, as well as for any CAT$(0)$-homology manifold, 
any discrete isometry group $\Gamma$ of $X$ is rigid, in the sense explained in Section \ref{subsection-isometries}.
\vspace{1mm}

For compact,  non-positively curved Riemannian orbifolds, the finiteness up to equivariant homotopy equivalence can be improved to finiteness up to equivariant diffeomorphisms: recall that an equivariant diffeomorphism between two Riemannian orbifolds $M=\Gamma \backslash X$, $M'=\Gamma' \backslash X'$ is  a diffeomorphism $F: X \rightarrow X'$ which is  equivariant with respect to some group isomorphism $\varphi: \Gamma \rightarrow \Gamma'$, i.e. $F(g \cdot x) = \varphi (g) \cdot F(x)$ for all $g\in \Gamma$ and all $x \in X$.
 
\begin{proof}[Proof of Corollary \ref{cor-intro-finiteness-of-manifolds}]
First notice that the splitting Theorem \ref{theo-intro-splitting} is true in the manifold category. 
Actually,  let $X= {\mathbb R}^n \times X_0$  be the De Rham decomposition of a Hadamard manifold $X$, where $X_0$ is the product of all  non-Euclidean factors.  The factor ${\mathbb R}^n $ coincides with the Euclidean factor splitted by $X$ as a CAT$(0)$-space (since the decomposition of a finite dimensional geodesic space into flat and irreducible factors is invariant by isometries, cp. \cite{FL06}). 
%@article{FL06,
%    author = "Foertsch, Thomas and Lytchak, Alexander",
%    title = "{The de Rham decomposition theorem for metric spaces}",
%    doi = "10.1007/s00039-008-0652-0",
%    journal = "Geom. Funct. Anal.",
%    volume = "18",
%    pages = "120--143",
%    year = "2008"
%}
Then, the ${\mathbb R}^k$ factor splitted by Theorem \ref{theo-intro-splitting} under an $\varepsilon$-collapsed action, with $\varepsilon \leq \sigma_0$, is isometrically immersed in the Euclidean de Rham factor ${\mathbb R}^n$ of the manifold decomposition of $X$; hence it is $C^\infty$-embedded as a submanifold, as well as its orthogonal complement ${\mathbb R}^{m-k}$.\linebreak
Then, also the factor $Y$ of Theorem \ref{theo-intro-splitting} is $C^\infty$-embedded in $X$, because $Y=X_0 \times {\mathbb R}^{m-k}$.
From this, we deduce that also the renormalization Theorem \ref{theo-bound-systole} holds in the manifold category: that is, every splitting $Y_i \times {\mathbb R}^{k_i}$ considered in the proof is smooth, and the resulting decomposition \eqref{eq-finalsplitting} yields a smooth Riemannian structure on $X'$ (such that $X'=X$ as a differentiable manifold).
Moreover, if $-\kappa^2 \leq k(X) \leq 0$, the new Riemannian manifold $X'$ satisfies the same curvature bounds 
(since the metric is dilated on the Euclidean factors by the constants $\frac{1}{L_i} >1$, as explained in Remark \ref{rmk-splitting}).\\
It then follows that  every $n$-dimensional Riemannian orbifold
 $M=\Gamma \backslash X$ with $-\kappa^2 \leq k(M) \leq 0$ and diam$(M) \leq D_0$ is equivariantly diffeomorphic to 
  a Riemannian orbifold $M'=\Gamma \backslash X'$ still satisfying  $-\kappa^2 \leq k(M') \leq 0$,  but with diameter bounded above by $\Delta_0$ and with sys$^\diamond (\Gamma, X') \geq s_0$; where the  constants $\Delta_0$ and $s_0$ only depend on $D_0$ and on the packing constants $(P_0,r_0)$ of $X$ (and then,  ultimately,  only on $\kappa$ and on the dimension $n$, by \eqref{eq-bishop}).\linebreak
 Moreover, by Theorem \ref{prop-vol-sys-dias}, the systole of $\Gamma$ acting on $X'$ is bounded below by $b_0 \cdot \min \{ \text{sys}^\diamond (\Gamma, X'), \varepsilon_0\}$ at some point $x'$, where again $b_0, \varepsilon_0$ only depend on $(P_0,r_0)$.
We  then deduce the finiteness of these orbifolds from Fukaya's \cite[Theorem 8.1]{Fuk86} (or from Cheeger's finiteness theorem, as completed in \cite{Pet84}, \cite{Yam85}, in the torsion-free case).
%Moreover, the quotients manifolds $M=\Gamma \backslash X$ and $M'=\Gamma \backslash X'$ are diffeomorphic (the identity $id: X \rightarrow X'$ is a $\Gamma$-equivariant diffeomorphism, by construction). So, every $N$-dimensional \linebreak nonpositively curved manifold $M$ with curvature $-\kappa^2 \leq k(M) \leq 0$ with diam$(X) \leq D_0$ is diffeomorphic to a nonpositively curved manifold $M'$ satisfying the same curvature bounds (as $X'$ is isometric to $X$), but having injectivity radius bounded below by $s_0$ and diameter bounded above by $\Delta_0$. These are  constants only depending on $D_0$ and on the packing constants $(P_0,r_0)$ of $X$, which ultimately depend only on $\kappa$ and on the dimension $N$, by \eqref{eq-bishop}.\linebreak
%We then conclude by using Cheeger's finiteness theorem, as completed in \cite{Pet84}, \cite{Yam85}.
\end{proof}

\bibliographystyle{alpha}
\bibliography{collapsing_22_01_24}

\end{document}